\documentclass[a4paper,11pt]{article}

\usepackage{amsmath,amssymb,amsthm}

\usepackage[table]{xcolor}

\usepackage{microtype}
\usepackage{comment}
\usepackage{longtable,hhline}
\usepackage[includeheadfoot,margin=3.5cm]{geometry}

\usepackage{tikz-cd}

\usepackage{url}
\usepackage{here}
\usepackage{enumitem}
\usepackage{a4wide}
\usepackage{hyperref}

\newtheorem{theorem}{Theorem}[section]
\newtheorem{lemma}[theorem]{Lemma}
\newtheorem{proposition}[theorem]{Proposition}
\newtheorem{corollary}[theorem]{Corollary}

\newtheorem{question}[theorem]{Question}
\newtheorem{notation}[theorem]{Notation}

\newtheorem{remark}[theorem]{Remark}
\newtheorem{definition}[theorem]{Definition}
\newtheorem{example}[theorem]{Example}

\newcommand{\ZZ}{\mathbb Z}

\newcommand{\PP}{\mathbb P}

\newcommand{\QQ}{\mathbb Q}

\newcommand{\Eff}{\operatorname{Eff}}
\newcommand{\Mov}{\operatorname{Mov}}

\newcommand{\Curv}{\operatorname{Curv}}

\newcommand{\rat}{\dashrightarrow}

\newcommand{\Bs}{\operatorname{Bs}}

\newcommand{\arrow}{\rightarrow}

\newcommand{\discrep}{\operatorname{discrep}}

\newcommand{\BL}{\operatorname{BL}}
\newcommand{\BJ}{\operatorname{BJ}}
\newcommand{\BS}{\operatorname{BS}}

\newcommand{\ourproduct}{\mathbb P^n \times \mathbb P^{n+1}}
\newcommand{\blowup}[3]{X^{#1,#2}_{#3}}
\newcommand{\ourproductblowup}[1]{X^{n,n+1}_{#1}}
\newcommand{\ourproductblowuptwo}{X^{n,n+1}_{n+2}}
\newcommand{\ourproductblowupthree}{X^{n,n+1}_{n+3}}

\title{Bilinear secants and birational geometry of blowups of $\PP^n \times \PP^{n+1}$}
\date{}

\author{Elisa Postinghel and Artie Prendergast-Smith}
\makeatletter
\newcommand{\subjclass}[2][1991]{%
  \let\@oldtitle\@title%
  \gdef\@title{\@oldtitle\footnotetext{#1 \emph{Mathematics Subject Classification.} #2}}%
}
\newcommand{\keywords}[1]{%
  \let\@@oldtitle\@title%
  \gdef\@title{\@@oldtitle\footnotetext{\emph{Key words and phrases.} #1.}}%
}
\makeatother

\begin{document}



\maketitle

\begin{abstract}

We introduce bilinear secant varieties and joins of subvarieties of  products of projective spaces, as a generalisation of the classical secant varieties and joins of projective varieties. 
We show that the bilinear secant varieties of certain rational normal curves of $\ourproduct$ play a central role in the study of the birational geometry of $\ourproductblowup{s}$, its blow-up in $s$ points in general position. We show that $\ourproductblowup{s}$ is log Fano, and we compute its effective and movable cones, for $s\le n+2$ and $n\ge1$ and for $s\le n+3$ and $n\le 2$, and  we compute the effective and movable cones of $X^{3,4}_6$. 
\end{abstract}

\section*{Introduction}

This paper studies the birational geometry of certain varieties obtained as blowups of products of projective spaces in general sets of points. A basic question is to determine which such varieties have good finiteness properties such as being log Fano or, more generally, a Mori dream space. Complete answers are known in several case: for blowups of a single projective space by Mukai \cite{Mukai04, Mukai05}, Castravet--Tevelev \cite{CT06}, and Araujo--Masssarenti \cite{AM16}, and for products of projective spaces of equal dimension by Mukai \cite{Mukai05} and Lesieutre-Park \cite{LP17}. (See Section 1 below for precise statements of these results.)

For products of projective spaces of unequal dimensions, much less is known. Let $X^{n,m}_s$ denote the blowup of $\PP^n \times \PP^m$ in $s$ general points.  The authors and Grange \cite{GPP21} proved the log Fano property for $X^{1,2}_s$ and $X^{1,3}_s$ with $s \leq 6$. Bolognesi--Massarenti--Poma \cite{bolognesi2023cox} did the same for $X^{1,n}_s$ with $s \leq n+1$, and also gave a presentation of the Cox ring in these cases.

In this paper, we consider the case of $\ourproductblowup{s}$.  If $s\le n+1$, then $\ourproductblowup{s}$ is toric, hence log Fano and a Mori dream space. If $n\ge5$ and $s\ge n+4$, then $\ourproductblowup{s}$ is not a Mori dream space, nor log Fano, because its effective cone is not finitely generated; see Section \ref{section-preliminaries}.
Therefore, besides a few cases in small dimensions (see Question \ref{question-LF-smalldim}), the only open cases for $n\ge 5$ are $s= n+2$ and $s=n+3$, which we consider in this paper.

Our first main result Theorem \ref{theoremXn,n+1,n+2-logFano} shows that $\ourproductblowup{s}$ is log Fano for $s \leq n+2$. An explicit log Fano pair is obtained with an effective divisor given as the union of two irreducible divisors, one pulled back from each factor. Along the way, in Theorems \ref{eff-cone-n+2} and \ref{mov-cone-n+2} we obtain a description of the effective and movable cones of divisors of $\ourproductblowuptwo$, by means of a technique called the \emph{cone method} developed in previous work with Grange \cite{GPP21}. Analogously to the case of one factor $\PP^n$ blown-up at $n+2$ points, the extremal rays of the effective cone are spanned by two types of divisors: exceptional divisors and pullbacks of fixed hyperplanes from both factors. 

We then consider the case $\ourproductblowupthree$: here there is a new ingredient. First, there is a unique rational curve, the \emph{distinguished rational curve}, of bidegree $(n,n+1)$ in $\ourproduct$ passing through $n+3$ general points in general position; see Corollary \ref{corollary-ratcurvefewerpoints}, part (a). This is analogous to the unique rational curve of degree $n$ passing through $n+3$ points in general position in $\PP^n$. 
Continuing with the analogy, just as the secant varieties of the rational normal curve of $\PP^n$ play a central role in the birational geometry of the blow-up of $\PP^n$ at $n+3$ points in general position --- for instance, when they have codimension 1, their proper transforms span extremal rays of the effective cone \cite{Mukai05,CT06,BDPn+3} --- so too do the \emph{bilinear secant varieties} of the distinguished rational curve of $\ourproduct$.

Bilinear secant varieties and joins of subvarieties of products of projective spaces are introduced in Section \ref{section-bilinear}. In particular,  bilinear secant varieties of the distinguished rational curve of $\PP^n\times\PP^{n+1}$ are described in Proposition \ref{proposition-bilinearsecantdim} as determinantal varieties. For these varieties we provide \emph{base locus lemmas}, see Propositions \ref{BLI for bilinear spans}, \ref{prop-bli-bisecant} and \ref{prop-bli-bilinearjoins},  which are useful when implementing the cone method to obtain a description of the effective and movable cones of divisors, as well as when computing discrepancies of log resolutions of pairs in order to prove the log Fano property.

The case of the threefold $X^{1,2}_4$ was covered completely in \cite{GPP21}. In Theorem \ref{theorem-x235-logfano}, we use the bilinear secant construction to prove that the variety $X^{2,3}_5$ is log Fano, and in Theorem \ref{theorem-effcone-x235} and Corollary \ref{corollary-X235movable} we describe its effective and movable cones of divisors. 
The next case is $X^{3,4}_6$, for which we obtain a description of the effective and movable cones in Theorem \ref{theorem-effcone-x346}. We note that the same analogy with the case of $\PP^n$ that works for $s=n+2$ points does not continue to work for $s=n+3$: in this case the extremal rays of the effective cone include exceptional divisors, pullbacks of extremal divisors from the factors, and bilinear secant varieties or joins, but also, somewhat unexpectedly, certain pencils of bidegree $(1,1)$. Finally, in Section \ref{section-Xn,n,n+3} we present some conjectures and questions about bilinear secant varieties in higher dimensions and their role in the birational geometry of $X^{n,n+1}_{n+3}$.

\subsection*{Acknowledgements}
The first author is a member of GNSAGA of INdAM (Italy); her research is funded by the European Union under the Next Generation EU PRIN 2022 \emph{Birational geometry of moduli spaces and special varieties}, Prot. n. 20223B5S8L --  CUP: E53D23005400001. The second author was supported by the Engineering and Physical Sciences Research Council [grant number EP/W026554/1].
\section{Preliminaries and background}\label{section-preliminaries}

We work throughout over an algebraically closed field of characteristic zero. We use standard definitions from birational geometry, to be found for example in \cite{KM98}. In particular we recall that a normal projective variety $X$ is {\it log Fano} or {\it Fano type} if there is an effective $\QQ$-divisor $\Delta$ such that $(X,\Delta)$ is klt and $-(K_X+\Delta)$ is ample; $X$ is called a {\it Mori dream space} if its Cox ring is a finitely generated algebra over the base field. Hu--Keel \cite{HK00} showed that Mori dream spaces enjoy the ``best possible'' properties from the point of view of minimal model theory: in particular, their effective and movable cones are rational polyhedral. Birkar--Cascini--Hacon--McKernan \cite{BCHM} showed that every log Fano variety is a Mori dream space. 

We will be interested in when certain blowups of products of projective spaces have these properties. A complete classification of blow-ups in points of products of copies of $\PP^n$ that are Mori dream spaces is known. In particular, it turns out that being log Fano and being a Mori dream space are equivalent properties, as proved  by Mukai \cite{Mukai05}, Castravet--Tevelev \cite{CT06}, Araujo--Massarenti \cite{AM16} and Lesieutre-Park \cite{LP17}. An important feature of these blow-ups is that a certain Weyl group acts on the space of divisors, preserving the effective and movable cones. By contrast, for products of factors with different dimensions, as far as we know there is no such Weyl group action, making the question of determining those that are Mori dream spaces or log Fano, and of computing their effective cones, more challenging. 

We will denote by $X^{n,m}_s$ the blow-up of $\PP^n\times\PP^m$ in $s$ points in general position.
 Let us summarise what is known about whether or not the space $X^{n,m}_s$ has the log Fano or Mori dream space properties:
 \begin{itemize}
 \item If $n \leq m$ and $s \leq n+1$ then $X^{n,m}_s$ is toric, hence log Fano.
 \item If $n \geq 5$ then $X^{n,0}_s$ is log Fano if and only if it is a Mori dream space if and only if $s \leq n+3$, by Mukai \cite{Mukai05}, Castravet--Tevelev \cite{CT06}, and Araujo--Massarenti \cite{AM16}. More generally, Mukai \cite{Mukai05} and Lesieutre-Park \cite{LP17} characterised those blowups of products of the form $(\PP^n)^m$ which are Mori dream spaces or log Fano.
 \item If $n \geq2$ and $X^{n,0}_s$ is not a Mori dream space, then neither is $X^{n,m}_s$ \cite[Corollary 6.9]{GPP21}. In particular, by the previous point, if $n \geq 5$ and $s>n+3$ then $X^{n,m}_s$ is not a Mori dream space for any $m$. 
  \end{itemize}

\subsection{Notation}\label{notation}
In this subsection we collect various notations that we use throughout the paper.
\begin{notation}\label{notation:classes}
We use the following notation for the varieties we work on, and numerical classes of divisors and curves on them:
\begin{itemize}
\item $X^{n,m}_s$: the blowup of the product $\PP^n \times \PP^m$ in  a set of $s$ general points $\{p_1,\ldots,p_s\}$.
\item $\pi_1 \colon X^{n,m}_s \arrow \PP^n$ and $\pi_2  \colon X^{n,m}_s \arrow \PP^m$: the natural projections.
\item $H_1$ and $H_2$: the pullback to $X^{n,m}_s$ of hyperplane classes via $\pi_1$ and $\pi_2$ respectively.
\item $l_1$ and $l_2$: the class of a line contained in a general fibre of $\pi_2$ and $\pi_1$ respectively. 
\item $E_i$ and $e_i$: the exceptional divisor of the blowup of $p_i$ and the class of a line in $E_i$ respectively. 
\end{itemize}
\end{notation}
The set $\{H_1,H_2,E_1,\ldots,E_s\}$ is a basis for the space $N^1(X^{n,m}_s)$ and the set $\{l_1,l_2,e_1,...,e_s\}$ is a basis for $N_1(X^{n,m}_s)$. We have the following pairing between the two given bases:
\begin{align*}
  H_i \cdot l_j= \delta_{ij}; &\quad H_i \cdot e_j = 0;\\
  E_i \cdot l_j = 0; &\quad E_i \cdot e_j = -\delta_{ij}.
\end{align*}

\begin{notation}\label{notation:fibres}
We use the following notations for the points that we blow up, their projections to factors, the linear spans of subsets of the points and fibres over them:
\begin{itemize}
\item $p_1,\dots,p_s$: $s$ point in general position in $\PP^n\times\PP^m$.
\item $p^j_i$: the projection of $p_i$ to the $j$th factor, for $i\in\{1,\dots,s\}$ and $j\in\{1,2\}$.
\item $L^j_I$: the linear span of $\{p^j_i:i\in I\}$ for 
a subset $I\subset\{1,\dots,s\}$ and for $j\in\{1,2\}$.
\item $\Pi^j_I$: the fiber over $L^j_I$ via the $j$th projection for 
a subset $I\subset\{1,\dots,s\}$ and for $j\in\{1,2\}$, i.e. $\Pi^1_I=L^1_I\times\PP^m$ and $\Pi^2_I=\PP^n\times L^2_I$.
\item  $A=\{\Pi^1_I: I\subset\{1,\dots,s\},1\le |I|\le n\}$: the set of all fibres over linear spans of sets of points projected to the first factor.
\item $B=\{\Pi^2_I: I\subset\{1,\dots,s\},1\le |I|\le m\}$: the set of all fibres over linear spans of sets of points projected to the second factor.
\item $C=\{L^1_I\times L^2_J: I_1,I_2\subset\{1,\dots,s\},1\le |I_1|\le n,  1\le |I_2|\le m\}$: the set of all the intersections of an element of $A$ and an element of $B$.
\end{itemize}
When no confusion is likely to arise, we will sometime,  abusing notation, use the same symbols for the above subvarieties of $\PP^n\times\PP^m$ and their strict transforms in $X^{n,m}_s$.
\end{notation}

\subsection{The nef cone of \texorpdfstring{$X^{n,m}_s$}{X{n,m}}}
We will start by showing that, in the cases we will consider, the nef cone of $X^{n,m}_s$ is easy to compute. The proof is based on the following lemma, whose statement and proof are modified and simplified versions of those of \cite[Lemma 2.6]{CLO16}. 
\begin{lemma}\label{lemma-clo}
  Define the following vectors in $\ZZ^{s+2}$:
  \begin{align*}
    l_1 &= \left(1,0,0,\ldots,0\right);\\
    l_2 &= \left(0,1,0,\ldots,0\right);\\
    e_i &= \left(0,0,0,\ldots,0,1,0,\ldots 0\right)
  \end{align*}
  where for $e_i$ we have $1 \leq i \leq s$ and the nonzero entry occurs in place $i+2$.
  
  Suppose that $v=(a_1,a_2,-b_1,\ldots,-b_s) \in \ZZ^{s+2}$ is an integer vector satisfying
  \begin{enumerate}
  \item[(a)] $a_i, \, b_j \geq 0$ for all $i$ and $j$;
  \item[(b)] $a_1+a_2 \geq \sum_{i=1}^s b_i$.
  \end{enumerate}
Then $v$ is a positive linear combination of the vectors $e_i$ and $l_j-e_i$. 
\end{lemma}
\begin{proof}
  The proof is by induction on $a=a_1+a_2$. If $a=1$, then $(a)$ and $(b)$ imply that $v=l_j-\sum_I e_i$ for a set $I$ of cardinality at most 1, hence is of the form claimed.

Now suppose that $a>1$ and $v$ satisfies the given inequalities. By swapping $l_1$ and $l_2$ and the $e_i$ we can assume without loss of generality that $a_1 \geq a_2$ and $b_1 \geq \cdots \geq b_s$. If $b_1=0$ then $v=a_1l_1+a_2l_2$ which is of the form claimed, so assume that $b_1>0$. Let $v'=v-(l_1-e_1)=(a_1-1,a_2,-b_1+1,-b_2,\ldots,-b_k)=(a_1',a_2,-b_1^\prime,\ldots,-b_s)$; by our assumption on $b_1$ this satisfies $(a)$; moreover $(a_1-1)+a_2 \geq (b_1-1)+b_2+\cdots b_s$, so the new vector $v\prime$ satisfies $(b)$ also. By induction $v^\prime$ is of the form claimed, hence so is $v=v^\prime + (l_1-e_1)$.   
\end{proof}
The lemma implies that the cone of curves $\Curv(\blowup{n}{m}{s})$ is generated by the $e_i$ and $l_j-e_i$ if and only if the class $H_1+H_2-\sum_{i=1}^s E_i$ is nef. We now show that condition holds in cases of interest for us.

\begin{proposition}\label{proposition-nefcone}
  Let $n, m \geq 2$ and $s \leq n+m$. Then the class $D= H_1+H_2-\sum_{i=1}^s E_i$ is nef on $X^{n,m}_s$. Therefore the cone of curves $\Curv({X^{n,m}_s})$ is spanned by the classes $e_i$ for $1 \leq i \leq k$ and the classes $l_j-e_i$ for $j=1,2$ and $1 \leq i \leq s$.
\end{proposition}
\begin{proof}
  The second statement follows from the first by using Lemma \ref{lemma-clo}. Indeed, if $C$ is an irreducible curve on $X^{n,m}_s$, then either $C$ is contained in one of the $E_i$ or it is a strict transform of an irreducible curve on $\PP^n\times\PP^m$. In the first case, the class of $C$ is a positive multiple of $e_i$. In the second case, assuming that $D$ is nef we have $D \cdot C \geq 0$, so the class of $C$ must satisfy the assumptions of Lemma \ref{lemma-clo}, and therefore this class is a positive combination of the $e_i$ and the $l_j-e_i$. 
  
  It remains to prove that the class $D$ is nef. We use induction on $n+m$. The base case is $n=m=2$, in which case our variety $\blowup{2}{2}{s}$ is the blowup of $\PP^2 \times \PP^2$ in at most 4 points. For 3 or fewer points, the variety is toric, hence its cone of curves is generated by classes of the form $e_i$ and $l_j-e_i$. Since $D$ is nonnegative on all these classes, it is nef in this case. For $s=4$, consider the decomposition
  \begin{align*}
    D &= \left( H_1-E_1-E_2 \right)+ \left(H_2-E_3-E_4 \right).
  \end{align*}
  Each summand is represented by an effective divisor isomorphic to $\blowup{1}{2}{2}$, the blowup of $\PP^1 \times \PP^2$ in 2 points, and any negative curve for $D$ must be contained in one of these two divisors. But the variety $\blowup{1}{2}{2}$ is toric, so as above $D$ is nonnegative on any curve contained in one of these two divisors. Hence $D$ is nef in this case. 

  For the inductive step, now consider the variety $\blowup{n}{m}{s}$ with $n+m>4$, and the class $D$ on it. Assume without loss of generality that $n \leq m$. We can decompose our class as
  \begin{align*}
    D &= \left( H_1-\sum_{i=1}^n E_i \right) + \left( H_2-\sum_{i=n+1}^s  E_i \right).
  \end{align*}
  By our hypothesis on $s$, both summands are represented by effective divisors, namely strict transforms of pullbacks of hyperplanes on $\PP^n$ and $\PP^m$. Call these divisors $D_n$ and $D_m$: again, any irreducible curve on which our class is negative must be contained in either $D_n$ or $D_m$. Note that $D_n$ is isomorphic to $\blowup{n-1}{m}{n}$, which is toric, so as above $D$ is nonnegative on all curves in $D_n$.  Next consider the divisor $D_m$, which is isomorphic to $\blowup{n}{m-1}{s-n}$. By our inductive hypothesis, we can assume that the cone of curves of $D_m$ is also generated by classes of the form $e_i$ and $l_j-e_i$. Again we know that $D$ is nonnegative on all such classes, so we are done. 
\end{proof}

\section{The distinguished rational curve}

In this section, we prove the existence of a ``distinguished" rational curve of degree $n$ through $n+3$ points in $\ourproduct$. This can be thought of as the analogue in this setting of the unique rational normal curve through $n+3$ points in $\PP^n$. We will see later that this curve and its ``bilinear secant varieties" play a key role in the birational geometry of $\ourproductblowupthree$, similarly to how the rational normal curve of $\PP^n$ and its secant varieties play a role in the birational geometry of $\PP^n$ blown up at $n+3$ points. For later use, we also record some related results on existence of rational curves of specified types.

 \begin{proposition}\label{prop-rational-curve} Given $s+2$ general points in $\PP^1 \times \PP^s$, there is a unique irreducible curve $C$ of bidegree $(1,s)$ passing through all the points. This curve is smooth and rational.
  \end{proposition}
  \begin{proof}
Let $\Gamma$ be a curve of bidegree $(1,s)$ in $\PP^1 \times \PP^s$ passing through a set of $s+2$ general points. We have $\Gamma \cdot F=1$ for any fibre $F$ of the projection $\PP^1 \times \PP^s \arrow \PP^1$, so $\Gamma$ must be smooth and the projection $p_1 \colon \Gamma \arrow \PP^1$ must be birational, hence an isomorphism. Note that if $p_2 \colon \Gamma \arrow \PP^s$ is the map to the second factor, then $\Gamma$ is precisely the graph of the morphism $\varphi = p_2 \circ p_1^{-1} \colon \PP^1 \arrow \PP^s$, whose image is a rational normal curve in $\PP^s$. The morphism $\varphi$ is therefore of the form $\varphi=\left[f_1,\ldots,f_{s+1}\right]$ where the $f_i$ are homogeneous forms of degree $s$ on $\PP^1$.
    
    To prove that a unique such $\Gamma$ exists, we can assume without loss of generality that the $s+2$ points are of the form
    \begin{align*}
      \left(p_1, [1,0,\ldots,0]\right),\ldots,\left(p_{s+1},[0,\ldots,1]\right), \, \left(p_{s+2},[1,1,\ldots,1]\right) 
    \end{align*}
    where the $p_i$ are pairwise distinct points in $\PP^1$. In terms of the morphism $\varphi=\left[f_1,\ldots,f_{s+1}\right]$ this means we have conditions
    \begin{align*}
      f_i(p_j) &=0 \ \text{ for all } 1 \leq i, \, j \leq s+1, \, i \neq j\\
        f_i(p_{s+2}) &= f_j(p_{s+2}) \ \text{ for all } i,j.
    \end{align*}
Given $s$ distinct points in $\PP^1$ there is a nonzero form of degree $s$ vanishing at all $s$ points and nowhere else, and this form is unique up to scalar. So the first set of conditions determine forms $f_1,\ldots,f_{s+1}$, each unique up to a scalar. Fixing a choice of $f_1$, the second set of conditions above then determine $f_2,\ldots,f_{s+1}$ uniquely; therefore the map $\varphi$, or equivalently the curve $\Gamma$, is uniquely determined. 
  \end{proof}
  \begin{remark}
 We remark that the curve  C of Proposition \ref{prop-rational-curve} plays a role in the birational geometry of the space $X^{1,n}_{n+2}$, the product $\PP^1\times\PP^n$ blown-up at $n+2$ points in general position. In fact the strict transform of $C$  spans an extremal ray of the Mori cone, see \cite[Proposition 3.2]{bolognesi2023cox}.
  \end{remark}

\begin{corollary} \label{corollary-ratcurvefewerpoints}
\begin{enumerate}
 \item[(a)] Given $n+3$ general points in $\ourproduct$, there is a unique irreducible curve of bidegree $(n,n+1)$ passing through all the points. This curve is smooth and rational.
 \item[(b)] Given $s \leq n+2$ general points in $\ourproduct$, the family of curves of bidegree $(n,n+1)$ passing through all the points sweeps out a Zariski-dense subset of $\ourproduct$.
 \item[(c)] Suppose $d \geq s$ and $k<d+2$. Given $k$ general points in $\PP^1 \times \PP^s$, the family of irreducible curves of bidegree $(1,d)$ passing through all the points sweeps out a Zariski-dense subset of $\PP^1 \times \PP^s$.
\end{enumerate}
\end{corollary}
\begin{proof}
For (a), projecting our given points to $\PP^n$ we get a general set of $n+3$ points in $\PP^n$, so there is a unique rational normal curve $R_n \subset \PP^n$ passing through all these points. Any irreducible curve $C$ of the given bidegree must project onto $R_n$, so must be contained in the subvariety $\pi_1^{-1}(R_n)$. This subvariety is isomorphic to $\PP^1 \times \PP^{n+1}$ and the isomorphism maps $C$ to a curve of bidegree $(1,n+1)$ passing through the images of the $n+3$ points. By Proposition \ref{prop-rational-curve} there is a unique such curve $C$, which must be smooth and rational.

Statement (b) follows immediately from (a).

To prove (c), choose $k$ general points in $\PP^1 \times \PP^d$. By Proposition \ref{prop-rational-curve} , it is again immediate that the family of irreducible curves of bidegree $(1,d)$ in $\PP^1 \times \PP^d$ passing through the points sweeps out a Zariski-dense subset. Now projecting away from a general $\PP^{d-s-1}$ in $\PP^d$ we get a family of irreducible curves of bidegree $(1,d)$ in $\PP^1 \times \PP^s$ through $k$ general points. 
\end{proof}
For convenience we introduce the following terminology for the curve whose existence is proved in part (a) of the previous Corollary: 
\begin{definition}\label{def-distinguishedcurve}
    Given $n+3$ general points in $\ourproduct$, the unique irreducible curve of bidegree $(n,n+1)$ passing through all the points is called the \emph{distinguished rational curve} determined by the points.
\end{definition}
\section{Bilinear secant varieties}\label{section-bilinear}
In this section, we define some generalisations of classical projective constructions to the setting of products of projective spaces. To keep things simple, we restrict ourselves to products with only two factors, since this will be the case of interest to us in later sections.

  \begin{definition}\label{definition-bilinearspan}
    Let $z_1,\ldots,z_k$ be points in in $\PP^n \times \PP^m$. Assume that the points are \emph{general}, by which we mean that the sets of their projections $\{z^1_1,\dots,z^1_k\}$ and $\{z^2_1,\dots,z^2_k\}$ to the first and second factors are linearly independent sets in $\PP^n$ and $\PP^m$ respectively.  Denote the linear spans of these two subsets by $L^1(z_1,\ldots,z_k) \subset \PP^n$ and $L^2(z_1,\ldots,z_k) \subset \PP^m$ respectively.

    The \emph{bilinear span} $\BL(z_1,\ldots,z_k) \subset \PP^n \times \PP^m$ of  $\{z_1,\ldots,z_k\}$ is defined as the product
    \begin{align*}
      L^1(z_1,\ldots,z_k) \times L^2(z_1,\ldots,z_k) \subset \PP^n \times \PP^m.
    \end{align*}
  \end{definition}
  
  \begin{definition}
    Let $X_1,\ldots,X_k$ be subvarieties of $\PP^n \times \PP^m$. The \emph{bilinear join} of $X_1,\ldots,X_k$ is defined as
    \begin{align*}
      \BJ(X_1,\ldots,X_k) = \overline{\bigcup_{z_i \in X_i} \BL(z_1,\ldots,z_k)}
    \end{align*}
    where the union is taken over $k$-tuples of general points (in the sense of Definition \ref{definition-bilinearspan}). 
  \end{definition}
Specialising our construction to the case where all the $X_i$ are equal, we get the corresponding notion of secant variety:
  \begin{definition}
    Let $X$ be a subvariety of $\PP^n \times \PP^m$. The \emph{$k$-th bilinear secant variety of $X$} is
    \begin{align*}
      \BS_k(X) = \BJ(\underbrace{X,\ldots,X}_{k \text{ times}} )
    \end{align*}
  \end{definition}
For this paper, the most important application of this construction is to the distinguished rational curve $C$ in $\ourproduct$. To understand the varieties $\BS_k(C)$ it is convenient to have a determinantal description similar to that for classical secant varieties. Up to projective transformations, we can assume that $C$ is the zero locus of the ideal of $2 \times 2$ minors of the matrix
  \begin{align}\label{eqn-M1}
    M_1 &=
    \begin{pmatrix}
      x_0 & x_1 & \cdots & x_{n-1} & y_0 & y_1 &\cdots & y_n \\
      x_1 & x_2 & \cdots & x_n & y_1 & y_2 & \cdots &y_{n+1}
    \end{pmatrix} \tag{$\ast$}
  \end{align}
where the $x_i$, respectively the $y_j$, are homogeneous coordinates on $\PP^n$, respectively $\PP^{n+1}$. We then have the following determinantal description of the bilinear secant varieties of $C$: 
  \begin{proposition}\label{proposition-bilinearsecantdim}
    The variety $\BS_k(C)$ is irreducible of dimension $\min \left\{ 3k-2, 2n+1 \right\}$. It is the zero locus in $\ourproduct$ of the ideal of maximal minors of the matrix 
    \begin{align*} \label{eqn-M2}
M_k &=
      \begin{pmatrix}
        x_0 & \cdots & x_{n-k} & y_0 & \cdots & y_{n-k+1} \\
        x_1 & \cdots & x_{n-k+1} & y_1 & \cdots & y_{n-k+2} \\
        & \vdots & & & \vdots & \\
        x_k & \cdots & x_n & y_k & \cdots & y_{n+1}\\
             \end{pmatrix} \tag{$\ast \ast$}
    \end{align*}
  \end{proposition} 
 \begin{proof}
   Catalano--Johnson \cite{CJ96} and De Poi \cite{dP96} showed that, in the projective space $\PP^{2n+2}$ with homogeneous coordinates $[x_0,\ldots,x_n,y_0,\ldots,y_{n+1}]$, the ideal of maximal minors of $M_k$ cuts out the $k$-th secant variety $S_k(\Sigma)$ of the rational normal surface scroll $\Sigma=\Sigma(n,n+1)$ of degree $2n+1$ defined by the $2\times2$ minors of $M_1$, and that $S_k(\Sigma)$ has the expected dimension $\operatorname{min}\left\{3k-1,2n+2 \right\}$.

   Now we consider the rational map
   \begin{align*}
     \varphi \colon \PP^{2n+2} &\dashrightarrow \ourproduct \\
     \left[x_0,\ldots,x_n,y_0,\ldots,y_{n+1} \right] &\mapsto \left(\left[x_0,\ldots,x_n\right],\left[y_0,\ldots,y_{n+1}\right]\right)
   \end{align*}
   After blowing up the disjoint linear spaces $\Lambda_n = \{ y_i =0 \mid i = 0,\ldots n+1\}$ and $\Lambda_{n+1} = \{ x_i =0 \mid i = 0,\ldots n\}$ inside $\PP^{2n+2}$, the map $\varphi$ extends to a morphism. For a point $p=\left(\left[x_0,\ldots,x_n\right],\left[y_0,\ldots,y_{n+1}\right]\right) \in \ourproduct$ the fibre of $\varphi$ over $p$ is the line
   \begin{align*}
    \left\{\left[tx_0,\ldots,tx_n,sy_0,\ldots,sy_{n+1}\right]\mid [t,s] \in \PP^1 \right\}
   \end{align*}
   Since $M_k$ is homogeneous with respect to both sets of variables, if such a line intersects $S_k(\Sigma)$ then it is contained in $S_k(\Sigma)$. Therefore the image $\varphi(S_k(\Sigma))$ is a variety of dimension $\dim S_k(\Sigma)-1 = \min\{3k-2,2n+1\}$, and evidently it is also defined by the maximal minors of $M_k$.

   Finally we need to show that $\varphi(S_k(\Sigma))=\BS_k(C)$ as defined above. It is clear that $\BS_k(C)$ is irreducible and its dimension is at most $\min\{3k-2,2n+1\}$, so by the previous paragraph it suffices to show that $\varphi(S_k(\Sigma)) \subset \BS_k(C)$. For $k=1$ there is nothing to prove since $C$ is by definition cut out by the maximal minors of $M_1$. Now let $k \geq 1$ and assume that $\varphi(S_k(\Sigma)) \subset \BS_k(C)$; we will prove that $\varphi(S_{k+1}(\Sigma)) \subset \BS_{k+1}(C)$.

   Consider a general point $p \in S_{k+1}(\Sigma)$. We can choose points $q \in S_k(\Sigma)$ and $r \in \Sigma$ such that $p$ lies on the line $L$ spanned by $q$ and $r$. By the inductive assumption $\varphi(q)$ is a point of $\BS_k(C)$ and $\varphi(r)$ is a point of $C$, and moreover $\varphi(L)$ is a curve of bidegree $(1,1)$ passing through $\varphi(q)$ and $\varphi(r)$. Any such curve is contained in the bilinear span $L(\varphi(q),\varphi(r))$ and hence in $\BJ(\BS_k(C),C)=\BS_{k+1}(C)$, as required.
  \end{proof}

\begin{proposition}\label{prop-secantmult}
Suppose $n=2\mod 3$ and $k=(2n+2)/3$, so that $\BS_k(C)$ is a divisor in $\ourproduct$. Then the multiplicity of $BS_k(C)$ at a general point of $C$ equals $k$.
\end{proposition}
\begin{proof}
Since $\BS_k(C)$ is a divisor by hypothesis, the matrix $M_k$ in (\ref{eqn-M2}) is square, and $\BS_k(C)$ is defined by the vanishing of its determinant. The partial derivatives of $\det(M_k)$ of a given order $d$ with respect to the $x_i$ and $y_j$ are linear combinations of $(k+1-d)\times(k+1-d)$ minors of $M_k$. Therefore, it suffices to prove that all minors of $M_k$ of size $(k+1)-(k-1)=2$ vanish along $C$. 

Since $C$ is defined by the vanishing of the $2 \times 2$ minors of the matrix $M_2$, it suffices to prove that the ideal of $2 \times 2$ minors of $C$ contains the ideal of $2 \times 2$ minors of $M_k$. For simplicity of notation, first consider $2 \times 2$ minors of $M_k$ that involve both sets of variables: such a minor has the form $x_i y_j - x_k y_l$ where $i<k$ and $i+j=k+l$. We have
\begin{align*}
    x_iy_j-x_ky_l &= \sum_{a=i}^{k-1} \left( x_a y_{i+j-a} -x_{a+1} y_{i+j-a-1} \right)
\end{align*}
and each summand on the right-hand side is a $2 \times 2$ minor of $M_2$. For $2 \times 2$ minors of $M_k$ involving only the $x_i$ or the $y_j$ we reduce to the case of classical catalecticants, where the result is well-known: see for example \cite[Proposition 9.7]{Harr92}.
\end{proof}

\section{Base locus inequalities} \label{section-bli}

In this section we will derive base locus inequalities, giving lower bounds for the multiplicity of containment of various distinguished subvarieties of $\ourproductblowup{s}$ in the base locus of a given divisor class. We fix the following form for an effective divisor class on $\ourproductblowup{s}$, where $s \leq n+3$:
\begin{equation}\label{general form divisor}
D=d_1H_1+d_2H_2-\sum_{i=1}^sm_iE_i\end{equation}

We start by recording some simple "effectivity" inequalities that must be satisfied by all effective divisors on $\ourproductblowup{s}$.
\begin{lemma}\label{lemma-effectivityinequality}
Let $D$ be an effective divisor class of the form \eqref{general form divisor}. Then 
\begin{enumerate}
    \item[(a)] for every $i \in \{1,\ldots s\}$ we have $m_i \leq d_1+d_2$;
    \item[(b)] for every subset $I \subseteq \{1,\ldots,s\}$ with $|I|\leq n+2$ we have
\begin{align*}
    \sum_{i \in I} m_i \leq nd_1+(n+1)d_2.
\end{align*}
\end{enumerate}
\end{lemma}
\begin{proof}
If $C$ is a curve class whose irreducible representatives sweep out a Zariski-dense subset of $\ourproductblowup{s}$, then we must have $D \cdot C \geq 0$ for all effective divisors classes $D$.

To prove (a) we apply this fact with $C=l_1+l_2-e_i$, the class of the proper transform of a curve of bidegree $(1,1)$ through $p_i$. 

To prove (b), we apply this with $C=nd_1+(n+1)d_2-\sum_{i \in I} e_i$. Irreducible curves of this class sweep out a Zariski-dense subset by Corollary \ref{corollary-ratcurvefewerpoints} (b).
\end{proof}
Next, the simplest base locus inequality concerns the exceptional divisors over our blown-up points.
\begin{lemma}\label{bli-exceptional}
Let $D$ be an effective divisor class of the form \eqref{general form divisor}. Then $D$ contains the exceptional divisor $E_i$ at least $\max\{0,-m_i\}$ times in its base locus.
\end{lemma}
\begin{proof}
If $m_i \geq 0$ there is nothing to prove. So assume that $m_i<0$. Then $D \cdot l_i =m_i <0$. Since irreducible curves of class $l_i$ cover $E_i$, this means that every divisor in the class $D$ contains $E_i$ as a component. Replacing $D$ by the effective class $D-E_i$ and repeating the argument gives the statement of the lemma.
\end{proof}

\subsection{Base locus inequalities for bilinear joins and secant varieties}
Next we give a basic recipe on how, given base locus lemmas for two subvarieties of $\ourproductblowup{s}$, we can obtain a base locus lemmas for their bilinear join. This will then be applied in a number of cases.

\begin{lemma}\label{recipe-BLI-bilinear joins}
 Let $z_1,z_2$ be two distinct points of $\ourproduct$ with distinct projections to either factor. 
 Let $D$ be an effective divisor on $\ourproductblowup{s}$ of the form \eqref{general form divisor} that has multiplicity $m_{z_i}$ at $z_i$, for $i=1,2$. (If $z_i=p_j$ for some $j$, we set $m_{z_i}:=m_j$). Then $D$ contains the strict transform of the bilinear span $\BL(z_1,z_2)$ at least $\max\{0,m_{z_1}+m_{z_2}-d_1-d_2\}$ times in its base locus.
\end{lemma}
\begin{proof}
We prove the case where $z_1,z_2\neq p_j$, for every $j=1,\dots,s$. The cases where one among $z_1,z_2$ or both belong to the set of points $\{p_1,\dots,p_s\}$ the proof is simpler and left to the reader.

Consider the blow-up $\phi_{z_1,z_2}: Y_{z_1,z_2}\to X$ of $X$ at $z_1,z_2$. For $i=1,2$, we call $E_{z_i}$ the exceptional divisor of $z_i$ and $e_{z_i}$ the line class of $E_{z_i}$. Under this choice, the strict transform of $D$ is 
$$D_{z_1,z_2}=\phi_{z_1,z_2}^\ast D-m_{z_1}E_{z_1}-m_{z_2}E_{z_2}.$$
Also, the class of the strict transform of a curve of bidegree $(1,1)$ that passes through $z_1$ and $z_2$ is 
$$C_{z_1,z_2}=l_1+l_2-e_{z_1}-e_{z_2}.$$ 
Since the irreducible curves $C_{z_1,z_2}$ sweep out the strict transform of $\BL(z_1,z_2)$, then $D$ contains $\BL(z_1,z_2)$ in its base locus at least 
$$
\max\{0,-D_{z_1,z_2}\cdot C_{z_1,z_2}\}
$$
times. This concludes the proof.
\end{proof}

We can apply Lemma \ref{recipe-BLI-bilinear joins} to a number of interesting subvarieties of $\ourproductblowup{s}$. 

\subsubsection{Base locus inequalities for bilinear spans of points}
We first of obtain a base locus lemma for the strict transform on $X$ of the bilinear spans of a collection of points $\ourproduct$.
Notice that the strict transform of the bilinear span of $n+1$ points, $\BL(p_{i_1},\dots,p_{i_{n+1}})$, has divisorial class $H_2-\sum_{j=1}^{n+1}E_{i_j}$.

\begin{proposition}\label{BLI for bilinear spans} Let $D$ be an effective divisor on $\ourproductblowup{s}$ of the form 
\eqref{general form divisor}.
\begin{enumerate}
    \item[(a)]
Take $p_{i_1},\dots,p_{i_k}$ to be $k$ among the $s$ points blown-up to obtain $\ourproductblowup{s}$. Then $D$ contains the strict transform of $\BL(p_{i_1},\dots,p_{i_k})$ in its base locus at least
\begin{align*}
\kappa_{\BL(p_{i_1},\dots,p_{i_k})}(D):=\max\{0,m_{i_1}+\cdots+m_{i_k}-(k-1)(d_1+d_2)\}
\end{align*}
times.
 \item[(b)]
Moreover, if $n+1\le s\le n+2$, then the 
 multiplicity of containment of $F_2:=H_2-\sum_{j=1}^{n+1}E_{i_j}$ in the base locus of D is exactly
 \begin{align*}
    \kappa_{F_2}(D):=\max\{0,m_{i_1}+\cdots+m_{i_{n+1}}-n(d_1+d_2)\}.
\end{align*}
\end{enumerate}
\end{proposition}
\begin{proof}
We prove the first statement by induction on $k$. The initial case $k=2$ is a particular instance of Lemma \eqref{recipe-BLI-bilinear joins}: $D$ contains the strict transform of $\BL(p_{i_1},p_{i_2})$ in its base locus at least 
$\kappa_{\BL(p_{i_1},p_{i_2})}(D) =\max\{m_{i_1}+m_{i_2}-d_1-d_2\}$ times.

We assume the statement holds for $k-1$ and we prove it for $k$. 
Observe that the bilinear span $\BL(p_{i_1},\dots,p_{i_k})$ is obtained as the union of the bilinear spans of a point $z$ of $\BL(p_{i_1},\dots,p_{i_{k-1}})$ and of $p_{i_{k}}$. 
By Lemma  \ref{recipe-BLI-bilinear joins}, the bilinear span of $z$ on $p_{i_{k}}$ is contained in the base locus of $D$ at least 
\begin{equation}\label{intermediate-mult}
   \max\{0,m_{z}+m_{i_{k}}-(d_1+d_2)\} 
\end{equation}
times, where $m_z$ is the multiplicity of $D$ at $z$, and so the number \eqref{intermediate-mult} is a lower bound for $\kappa_{\BL(p_{i_1},\dots,p_{i_k})}(D)$.
If $m_z=0$, since  $m_{i_{k}}\le d_1+d_2$ by the effectivity assumption on $D$, then \eqref{intermediate-mult} is zero and the statement holds trivially. 
Assume now that $m_z>0$.
By the induction hypothesis, we have 
\begin{equation}\label{intermediate-mult-2}
m_{z}\ge \kappa_{\BL(p_{i_1},\dots,p_{i_{k-1}})}(D)=\max\{0,m_{i_1}+\cdots+m_{i_{k-1}}-(k-2)(d_1+d_2)\}.
\end{equation}
Combining formulae \eqref{intermediate-mult} and \eqref{intermediate-mult-2}, we obtain the first statement.

We now prove the second statement. Let $F_2$ be the divisorial class of the strict transform of the bilinear span of $n+1$ points. Modulo reordering the points, we may assume that $\{i_1,\dots, i_{n+1}\}=\{1,\dots, {n+1}\}$, so that $F_2=H_2-\sum_{j=1}^{n+1}E_i$. Set $$K_{F_2}(D):=\sum_{i=1}^{n+1}m_i-n(d_1+d_2),$$ so that we can write $\kappa_{F_2}(D)=\max\{0,K_{F_2}(D)\}$.

We claim that if $K_{F_2}(D)\le0$ or, equivalently, if $\kappa_{F_2}(D)=0$, then $F_2$ is not contained in the base locus of $D$.
Accepting the claim and observing that the integer $K_{F_2}(D)$ drops by one after subtracting $F_2$, i.e. $K_{F_2}(D-F_2)=K_{F_2}(D)-1$, we conclude that the multiplicity of containment of $F_2$ in $D$ is exactly  $\kappa_{F_2}(D)$.

We now prove the claim. If $s< n+1$
there is nothing to prove. If $s=n+1$, we may consider $D$ as a divisor on $X^{n,n+1}_{n+2}$ by setting  $m_{n+2}=0$. Therefore it is enough to prove the statement for $s=n+2$, which we now do.

Since $|D+E_i|=|D|+E_i$, where $E_i$ is an exceptional divisor, it is enough to prove the statement for effective divisors $D$ with $K_{F_2}(D)=0$. 
By Corollary \ref{corollary-ratcurvefewerpoints} (b), the curve class of degree $(n,n+1)$ passing through all $n+2$ points is moving. Therefore, since $D$ is effective, its intersection with such a curve class must be non-negative. This gives the following inequality.
$$
\sum_{i=1}^{n+2}m_i- nd_1-(n+1)d_2\le0.
$$
This, together with the assumption $K_{F_1}(D)=0$, implies that $m_{n+2}\le d_2$.

Now, the claim will be proved  adapting to this case an idea used in \cite[Lemma 4.24]{CT06} to obtain a decomposition of an effective divisor on $X^{n,0}_{n+2}$ as a sum of hyperplanes; because of its technical nature, we will  illustrate this procedure in Example \ref{example-procedure-CT} below.

We consider a table with $d_2+d_1$ rows and $n+1$ columns. We fill the entries of the columns each with an exceptional divisor starting from the first column and filling it top to bottom, then filling the second column etc. as follows: start by writing  $m_1$ times $E_1$, then $m_2$ times $E_2$, etc. until we will have placed $E_{n+1}$ exactly $m_{n+1}$ times. Since $K_{F_2}(D)=0$, we will have filled all cells of the first $n$ columns. We then consider the last column and we fill the entries on the first $d_1$ rows with zeroes and the entries of the bottom $d_2$ rows with the exceptional divisor $E_{n+2}$ exactly $m_{n+2}$ times and then with $d_2-m_{n+2}\ge0$ zeroes. 
Since $D$ is effective, we have $m_i\le d_1+d_2$; this condition guarantees that in each row there are no repetitions. This procedure will be illustrated in an example below.

This yields a decomposition of $D$ as a sum of $d_1$ divisor classes of type $H_1-\sum_{k=1}^nE_{i_k}$ and $d_2$ divisor classes of type $H_2-\sum_{k=1}^{s}E_{i_k}$ with, $s\in\{n,n+1\}$, whose sum is linearly equivalent to $D$. 
By construction, $F_2$ does not appear in this list, hence it is not a divisorial component of the base locus of $D$. This completes the proof. 
\end{proof}

\begin{example}\label{example-procedure-CT}
 The procedure adopted in the second part of the proof of Proposition \ref{BLI for bilinear spans} can be illustrated as follows, where, for this purpose, we choose $n=4$ and $D=2H_1+3H_2-5E_1-5E_2-4E_3-3E_4-3E_5-2E_6$:

$$
\left[
\begin{array}{l|l|l|l||l}
E_1&E_2&E_3&E_4&0\\
E_1&E_2&E_3&E_4&0\\
\hline
E_1&E_2&E_3&E_5&E_6\\
E_1&E_2&E_3&E_5&E_6\\
E_1&E_2&E_4&E_5&0\\
\end{array}
\right]\quad \rightarrow \quad 
\left. \begin{array}{ll}
&H_1-(E_1+E_2+E_3+E_4)\\
+&H_1-(E_1+E_2+E_3+E_4)\\
+&H_2-(E_1+E_2+E_3+E_5+E_6)\\
+&H_2-(E_1+E_2+E_3+E_5+E_6)\\
+&H_2-(E_1+E_2+E_4+E_5)
\end{array}\right\}
 = D
$$
\end{example}

\subsubsection{Base locus inequalities for bilinear secant varieties}
Now, assuming $s=n+3$, we give a base locus lemma for the bilinear secant varieties of the rational curve of degree $(n,n+1)$ passing through $p_1,\dots,p_{n+3}$, whose class is 
$$C:=nl_1+(n+1)l_2-\sum_{i=1}^{n+3}e_i.
$$

\begin{proposition} \label{prop-bli-bisecant}
Let $\BS_k(C)\subset X^{n,n+1}_{n+3}$ be the $k$th bilinear secant of $C$. The $D$, an effective divisor of the form \eqref{general form divisor}, contains $\BS_k(C)$ in its base locus at least
$$
\kappa_{\BS_k(C)}(D):=\max\{0,k\sum_{i=1}^{n+3}m_i-(nk+k-1)d_1-(nk+2k-1)d_2\}
$$ times.
\end{proposition}
\begin{proof}
We prove the statement by induction on $k$.
The initial case is $k=1$ and we have that the multiplicity of $D$ along $\BS_1(C)=C$ is 
$$
\kappa_C(D)=\max\{0,-D\cdot C\}= \max\{\sum_{i=1}^{n+3}m_i-nd_1-(n+1)d_2\}.
$$

We assume the statement for $k-1$ and we prove it for $k$.
By definition, $\BS_k(C)$ is the bilinear join of $\BS_{k-1}(C)$ and $C$. In particular  $\BS_k(C)$ is the Zariski closure of the union of the linear spans $L(z_1,z_2)$ with $z_1\in BS_{k-1}(C)$ and $z_2\in C$.
We will prove that the multiplicity of $D$ along the bilinear span $\BL(z_1,z_2)$ is at least the claimed one $\kappa_{\BS_k(C)}(D)$. Since the multiplicity is semicontinuous, we will be able conclude that the same holds also for the boundary points of $\BS_k(C)$.

We now prove the claim.
Assume that $z_1,z_2$ are general points. 
Then 
\begin{align*} 
m_{z_1}&\ge\kappa_{\BS_{k-1}(C)}(D)\\
&=\max\{0,(k-1)\sum_{i=1}^{n+3}m_i-(n(k-1)+k-2)d_1-(n(k-1)+2k-3)d_2\}
\end{align*}
by the induction hypothesis, while 
$m_{z_2}\ge \kappa_C(D)$, by the initial case.

By Lemma \ref{recipe-BLI-bilinear joins}, we have that the multiplicity of $D$ along $\BL(z_1,z_2)$ is at least 
$$
\max\{0,m_{z_1}+m_{z_2}-d_1-d_2\}
$$
 If $m_{z_1}=0$ or $m_{z_2}=0$, then the above number is zero by the effectivity condition $m_{z_i}\le d_1+d_2$, hence the statement follows trivially. If $m_{z_1}, \, m_{z_2}>0$, the statement follows from a simple calculation that uses the three above formulas. \end{proof}

\subsubsection{Base locus inequalities for bilinear joins}
In a similar fashion, we can prove a base locus lemma for any bilinear join between a bilinear secant variety and the bilinear span of a set of points.

\begin{proposition}\label{prop-bli-bilinearjoins}
The bilinear join  of $\BS_k(C)$ and the bilinear span $\BL(p_{i_1},\dots,p_{i_l})$, which is denoted by $\BJ(\BS_k(C),\BL(p_{i_1},\dots,p_{i_l}))\subset X^{n,n+1}_{n+3}$, 
is contained  in the base locus of an effective divisor $D$ of the form \eqref{general form divisor} at least 
\[
\kappa_{\BJ(\BS_k(C),\BL(p_{i_1},\dots,p_{i_l}))}(D):=\max\{0, k\sum_{i=1}^{n+3}m_i+\sum_{j=1}^lm_{i_j}-(nk+k+l-1)d_1-(nk+2k+l-1)d_2\}
\] times. 
\end{proposition}
\begin{proof}
By definition $\BJ(\BS_k(C),\BL(p_{i_1},\dots,p_{i_l}))$ is covered by the bilinear spans of a point $z_1$ of $\BS_k(C)$ and a point $z_z$ of $\BL(p_{i_1},\dots,p_{i_l})$.
We introduce the following integers:
\begin{align*}
K_{\BJ(\BS_k(C),\BL(p_{i_1},\dots,p_{i_l}))}(D)&:= k\sum_{i=1}^{n+3}m_i+\sum_{j=1}^lm_{i_j}-(nk+k+l-1)d_1-(nk+2k+l-1)d_2\\
K_{\BS_k(C)}(D) &:= k\sum_{i=1}^{n+3}m_i-(nk+k-1)d_1-(nk+2k-1)d_2\\
K_{\BL(p_{i_1},\dots,p_{i_l})}(D) &:= 0, \, m_{i_1}+\cdots+m_{i_l}-(l-1)(d_1+d_2).
\end{align*}
By Lemma \ref{prop-bli-bisecant} and Lemma  \ref{recipe-BLI-bilinear joins}, the multiplicities $m_{z_1}$ and $m_{z_2}$ of $D$ at $z_1$ and $z_2$ respectively satisfy
\begin{align*}
m_{z_1} &\geq \kappa_{\BS_k(C)}(D) = \max\{0, \, K_{\BS_k(C)}(D)2\},\\
m_{z_2} &\geq \kappa_{\BL(p_{i_1},\dots,p_{i_l})}(D) = \max\{0, K_{\BL(p_{i_1},\dots,p_{i_l})}(D)\}.
\end{align*}
Moreover, it is easy to check that $$K_{\BJ(\BS_k(C),\BL(p_{i_1},\dots,p_{i_l}))}(D)=K_{\BS_k(C)}(D)+K_{\BL(p_{i_1},\dots,p_{i_l})}-(d_1+d_2).$$

Note that by the effectivity hypothesis we must have $m_{z_1}, \, m_{z_2} \leq d_1+d_2$. In particular this implies that $K_{\BS_k(C)}(D) \leq d_1+d_2$ and  $K_{\BL(p_{i_1},\dots,p_{i_l})}(D) \leq d_1+d_2$.

If $\kappa_{\BS_k(C)}(D)=0$ or $\kappa_{\BL(p_{i_1},\dots,p_{i_l})}(D)=0$ the statement follows trivially from the effectivity hypothesis which implies that $m_{z_1}, \, m_{z_2}\le d_1+d_2$, hence the integer in the statement is zero.
Assuming that $m_{z_1},m_{z_2}>0$, we have 
\begin{align*}
m_{z_1}+m_{z_2}-(d_1+d_2)
&=k\sum_{i=1}^{n+3}m_i+\sum_{j=1}^lm_{i_j}-(nk+k+l-1)d_1-(nk+2k+l-1)d_2\\
&=K_{\BJ(\BS_k(C),\BL(p_{i_1},\dots,p_{i_l}))}(D)
\end{align*}
Using Lemma \ref{recipe-BLI-bilinear joins}, we conclude.
\end{proof}


\subsection{Base locus inequalities for pulled-back divisors}

As in Notation \ref{notation:fibres}, $p_1^j,\dots,p_s^j$ are the images of the  points $p_1,\dots,p_s\in\PP^n\times\PP^{n+1}$ via the natural projection  to the $j$th factor, for $j\in\{1,2\}$, and  $\pi_j:X^{n,n+1}_s\to\PP^{n-1+j}$ is the composition of the blow-up map with the projection to the $j$th factor. Recall that  $\Pi_{\{i\}}^j$ denotes the strict transforms under the blow-up map of the fibre of $p_i^j$ via the projection map to the $j$th factor, for $i=1,\dots,s$. 
 Each $\Pi_{\{i\}}^j\subseteq X^{n,n+1}_s$ is isomorphic to a $\PP^{n+2-j}$ blown-up at a point.
Let $\overline{\phi}_j:\overline{X^{n,n+1}_s}^j\to X^{n,n+1}_{s}$ be the blow-up along all $\Pi_i^j$'s and denote with $E_{\Pi^j_{\{i\}}}$ the corresponding exceptional divisors.

Now fix $j=1$.
The composition $\pi_1\circ\overline{\phi}_1:\overline{X^{n,n+1}_s}^1\to\PP^n$ factors through $X^{n,0}_{s}$,  the blow-up of $\PP^n$ at the points $p_i^1$, by the universal property of the blow-up.
Let $\overline{\pi}_1: \overline{X^{n,n+1}_{s}}^1 \to X^{n,0}_{s}$ be the corresponding map. These maps fit in the following diagram:
\[\begin{tikzcd}
\overline{X^{n,n+1}_s}^1  \arrow{r}{\overline{\phi}_1} 
\arrow{rd}  \arrow{d}{\overline{\pi}_1}
& X^{n,n+1}_{s} \arrow{d}{\pi_1}  \\
 X^{n,0}_{s} \arrow{r}{} & \PP^n
\end{tikzcd}
\]
For $j=2$, via a similar construction, we obtain the map $\overline{\pi}_2: \overline{X^{n,n+1}_{s}}^2 \to X^{0,n+1}_{s}$.

We will say that a divisor $D$ on $X^{n,n+1}_s$ is a \emph{pull-back} from the $j$th factor if it is of the form:
$$
d_jH_j-\sum_{i=1}^s m_i E_i.
$$
We will be interested in base locus inequalities for fixed pull-back divisors. 

Let $D$ be a general effective divisor on $X$ of class \eqref{general form divisor}. Its strict transform on $\overline{X^{n,n+1}_{s}}^j$ is
\begin{equation}\label{strict-transform-on-overlineX} 
\overline{D}^j=d_1H_1+d_2H_2-\sum_{i=1}^sm_iE_i-\sum_{i=1}^s\max\{0,m_i-d_{3-j}\}E_{\Pi_{\{i\}}^j} 
\end{equation}
where, by abusing notation, we denote by $E_i$ both an exceptional divisor on ${X}^{n,n+1}_s$ and its pull-back on $\overline{X^{n,n+1}_{s}}^j$.

Now let $\mathcal{H}^{j}$ be the strict transform on $\overline{X^{n,n+1}_{s}}^j$ of a general fibre of the projection $\pi_{3-j}$: it is a copy of $\PP^{n-1+j}$ blown-up at $s$ points, and the exceptional divisors are $E_{\Pi^j_{\{i\}}}|_{\mathcal{H}^{j}}$.

\begin{lemma}\label{bli for pullbacks}
Let $D$ be an effective divisor on $X$. Let $j\in\{1,2\}$ and let $F_j$ be a reduced and irreducible pull-back divisor of the form $F_j=d_jH_j-\sum_{i=1}^sm_iE_i$. Then the multiplicity of containment of $F_j$ in the base locus of $D$ is at least the multiplicity of containment of $\overline{F_j}|_{\mathcal{H}^j}$ in the base locus of $\overline{D}^j|_{\mathcal{H}^j}$.

\end{lemma}
\begin{proof}
Write $m_{F_j}$ for the multiplicity of containment of $F_j$ in the base locus of $D$. Fix an effective divisor $D' \in |D|$ such that $D'$ has multiplicity $m_{F_j}$ along $F_j$. For a general choice of $\mathcal{H}^j$ the divisor $\overline{D'}^j|_{\mathcal{H}^j}$ has multiplicity $m_{F_j}$ along $\overline{F_j}|_{\mathcal{H}^j}$. Therefore the multiplicity of containment of $\overline{F_j}|_{\mathcal{H}^j}$ in the base locus of $\overline{D}^j|_{\mathcal{H}^j}$ is at most $m_{F_j}$.
\end{proof}

We can now apply Lemma \ref{bli for pullbacks} to obtain specific base locus lemmas for fixed pull-back divisors.

\subsubsection{Pull-back hyperplanes}
As a first application,  we obtain a lower bound to the multiplicity of containment of fixed divisors of degree $(1,0)$ in the base locus of an effective divisor. As we will see, this is the in fact the exact multiplicity of containment.

\begin{proposition}\label{BLI pulledback hyperplanes}
Let $D$ be an effective divisor on $X^{n,n+1}_s$.
\begin{enumerate}
 \item[(a)]
 If $s\ge n$, then the  pulled-back fixed hyperplane $F_1=H_1-\sum_{k=1}^nE_{i_k}$ is contained is the base locus of $D$ at least 
 \begin{equation}\label{bli-formula-F1}
\kappa_{F_1}(D):=\max\{0,\sum_{j=1}^n m_{i_j}-(n-1)d_1-nd_2\}
 \end{equation}
 times. 

 \item[(b)]
If $n\le s\le n+2$, then the multiplicity of containment of $F_1$ in the base locus of $D$ is exactly equal to \eqref{bli-formula-F1}.
\end{enumerate}
 \end{proposition}

\begin{proof}[Proof of Proposition \ref{BLI pulledback hyperplanes}]
Given an effective divisor $D$, let us introduce the following convenient notation:$$K_{F_1}(D):=\sum_{j=1}^n m_{i_j}-(n-1)d_1-nd_2.$$
We first of all show that $F_1$ is contained in the base locus of $D$, for any number of points $s$,  at least $\kappa_{F_1}(D)=\max\{0,K_{F_1}(D)\}$ times. Assume that $m_{i_j}\le d_2$ for some $j\in\{1,\dots,n\}$ and, for the sake of simplicity,  let us set $j=n$, after reordering indices if necessary. We obtain 
$\sum_{j=1}^n m_{i_j}-(n-1)d_1-nd_2
\leq\sum_{j=1}^{n-1} m_{i_j}-(n-1)(d_1+d_2)\le0$, where the last inequality follows from  the effectivity condition $m_{i_j}\le d_1+d_2$ that must hold for every index $i_j$. Hence $\max\{0,K_{F_1}(D)\}=0$, therefore the statement holds trivially.

We will now assume that   $m_{i_j}>d_2$ for every $j\in\{1,\dots,n\}$.
Let $h,e_1,\dots,e_s$ be a basis for the Picard group of $\mathcal{H}_1$ obtained as follows:
\begin{equation}\label{Picard-basis-of-fibre}
h=H_1|_{\mathcal{H}_1}, \quad e_i=E_{\Pi^1_{\{i\}}}|_{\mathcal{H}_1}
\end{equation}
Using \eqref{strict-transform-on-overlineX} and \eqref{Picard-basis-of-fibre}, we obtain:
\begin{align*}
{\overline{F}^1_1}|_{\mathcal{H}_1}&=h-\sum_{j=1}^ne_{i_j}
\\
{\overline{D}^1}|_{\mathcal{H}_1}&=d_1h-\sum_{i=1}^s\max\{0,m_i-d_2\}e_{i}.
\end{align*}
That $\max\{0,k_{F_1}(D)\}$ is a lower bound for the multiplicity of containment of $F_1$ in the base locus of $D$ now follows from Lemma \ref{bli for pullbacks} and \cite[Lemma 2.1]{BDP-TAMS}. 

We now prove the second part of the statement, i.e. that for $s\le n+2$, then the integer $\max\{0,K_{F_1}(D)\}$  is the exact multiplicity of containment of $F_1$ in the base locus of $D$.

We claim that if $K_{F_1}(D)\le0$, then $F_1$ is not contained in the base locus of $D$. Accepting the claim and observing that the integer $K_{F_1}(D)$ drops by one after subtracting $F_1$, i.e.$K_{F_1}(D-F_1)=K_{F_1}(D)-1$, we can conclude that the multiplicity of containment of $F_1$ in $D$ is exactly the number $\max\{0,K_{F_1}(D)\}$.

We now prove the claim. If $s<n+1$ there is nothing to prove. If $s=n+1$, we may consider $D$ as a divisor on $X^{n,n+1}_{n+2}$ by adding a point with multiplicity zero. Therefore is is enough to prove the statement for $s=n+2$.

We prove the claim for $s=n+2$. Modulo reordering the points, we may assume that $\{i_1,\dots,i_n\}=\{1,\dots,n\}$. Since $|D+E_i|=|D|+E_i$, where $E_i$ is an exceptional divisor, it is enough to prove the statement for divisors $D$ with $K_{F_1}(D)=0$. 
By Corollary \ref{corollary-ratcurvefewerpoints}(b), the curve class of degree $(n,n+1)$ passing through all $n+2$ points is moving. Therefore, since $D$ is effective, its intersection with such a curve class must be non-positive. This gives the following inequality.
$$
\sum_{i=1}^{n+2}m_i- nd_1-(n+1)d_2\le0.
$$
This, together with the assumption $K_{F_1}(D)=0$, implies that $m_{n+1}+m_{n+2}\le d_1+d_2$.

Following the same idea as in the second part of the proof of Proposition \ref{BLI for bilinear spans}, we consider a table with $d_2+d_1$ rows and $n+1$ columns. Leaving the $n$-th entry of each of the bottom $d_1$ rows blank, we may fill  the remaining $(n-1)d_1+nd_2$ entries each with an exceptional divisor starting from the first column and filling it top to bottom, then filling the second column, and so on, as follows: start by writing  $m_1$ times $E_1$, then $m_2$ times $E_2$, etc., up to $m_{n}$ times $E_n$. Since $K_{F_1}(D)=0$, at this point we will have filled all cells of the first $n$ columns. We then fill the $n+1$-st column with $m_{n+1}$ times $E_{n+1}$ and $m_{n+2}$ times $E_{n+2}$. The condition $m_1\le d_1+d_2$ guarantees that  on each row there are no repetitions. 

This yields a decomposition of $D$ as a sum of $d_1$ divisor classes of type $H_1 -\sum_{k=1}^nE_{i_k}$ and $d_2$ divisor classes of type $H_2-\sum_{k=1}^{n+1}E_{i_k}$s. By construction, the sum  of these $d_1+d_2$ divisors is linearly equivalent to $D$. Moreover none of the summand equals $F_1$. Hence $F_1$ is not  a fixed component of $|D|$.
\end{proof}

\begin{remark}
Applying Lemma \ref{bli for pullbacks} to the pullback class of degree $(0,1)$ given by $F_2=H_2-\sum_{k=1}^{n+1}E_{i_k}$  yields the formula 
$\max\{0,\sum_{j=1}^{n+1} m_{i_j}-(n+1)d_1-nd_2\}.$
This formula however is weaker  than that obtained as an application of Proposition \ref{BLI for bilinear spans}, which is $\max\{0,\sum_{j=1}^{n+1} m_{i_j}-n(d_1+d_2)\},$
so we consider it redundant.
\end{remark}

\subsubsection{Pull-back secant varieties and joins from \texorpdfstring{$X^{n,0}_{n+3}$}{X{n,0}{n+3}}}
Assume that $s=n+3$. Then for every integer $t\ge0$ and every subset $I\subset\{1,\dots, n+3\}$ of cardinality $|I|=n-2t$, we can consider the pulled-back divisor
\begin{equation}\label{pulledback-cones}
F_{t,I}:=(t+1)H_1-(t+1)\sum_{i\in I}^{n+1} E_i-t\sum_{i\notin I}E_i.
\end{equation}
Its image via $\pi_1$ is the cone over the $t$-secant variety to the rational normal curve of degree $n$ of $\PP^n$ passing through $p^1_1,\dots,p^1_{n+3}$, with vertex the linear span of the points $p^1_i$ with $i\in I$. We refer to \cite{BDPn+3} for more details about these divisors.

If we consider $X^{n,0}_{n+3}$, the blow-up $\PP^n$ at $p^1_1,\dots,p^1_{n+3}$, the strict transforms of such cones, along with the exceptional divisors, generate the effective cone \cite{BDPn+3,CT06}.
Therefore, the divisors of the form \eqref{pulledback-cones} and $E_i$, $i\in\{1,\dots,n+1\}$ are all fixed extremal rays of the effective cone of $X^{n,n+1}_{n+3}$ and they span the facet $\Eff(X^{n,n+1}_{n+3})\cap \{d_2=0\}$ of such cone. For all these divisors we are able to give a base locus lemma.
\begin{proposition}\label{prop-bli-pulledbackcone}
Let $D$ be an effective divisor on $X^{n,n+1}_{n+3}$. Let $F_{t,I}$ be a pulled-back divisor of the form \eqref{pulledback-cones}, for $t\ge0$, $I\subset\{1,\dots,n+3\}$ with $|I|=n-2t\ge0$.
Then $F_{t,I}$ is contained in the base locus of $D$ at least $$
\kappa_{F_{t,I}}(D):=\max\{0,(t+1)\sum_{i\in I}m_i+t\sum_{i\notin I}m_i-((n+1)t+|I|-1)d_1-((n+3)t+|I|)d_2\}
$$
times.
\end{proposition}
\begin{proof}
Since the case $t=0$ is covered in Proposition \ref{BLI pulledback hyperplanes}, we will assume that $t\ge1$.
Assume that $m_j\le d_2$ for $j\in I$. We obtain
\begin{align*}
& (t+1)\sum_{i\in I}m_i+t\sum_{i\notin I}m_i\\
=& \, t\sum_{i\ne j}m_i+\sum_{i\in I\setminus\{j\}}m_i+(t+1)m_j\\
\le& \, t(nd_1+(n+1)d_2)+(|I|-1)(d_1+d_2)+(t+1)d_2\\
=& \,(nt+|I|-1)d_1+((n+2)t+|I|)d_2\\
\le& \, ((n+1)t+|I|-1)d_1+((n+3)t+|I|)d_2
\end{align*}
where the first inequality follows from the effectivity conditions $m_i\le d_1+d_2$ and
$\sum_{k=1}^{n+2}m_i\le nd_1+(n+1)d_2$ proved in Lemma \ref{lemma-effectivityinequality} (a) and (b) and from the assumption $m_j\le d_2$, while the other equalities and the second inequality follow from simple calculation.
In this case the statement holds trivially.

Assume that $m_j\le d_2$ for $i\notin I$. We obtain
\begin{align*}
& (t+1)\sum_{i\in I}m_i+t\sum_{i\notin I}m_i\\
=& \, t\sum_{i\ne j}m_i+\sum_{i\in I\setminus\{j\}}m_i+tm_j\\
\le& \, t(nd_1+(n+1)d_2)+(|I|-1)(d_1+d_2)+td_2\\
=& \, (nt+|I|-1)d_1+((n+2)t+|I|-1)d_2\\
\le& \, ((n+1)t+|I|-1)d_1-((n+3)t+|I|)d_2
\end{align*}
and conclude similarly.

We now assume that $m_i> d_2 $ for every $i\in\{1,\dots,n+3\}$.
Let $\mathcal{H}_1$ be as above. Using \eqref{strict-transform-on-overlineX} and \eqref{Picard-basis-of-fibre}, we obtain:
$$
{\overline{F_{t,I}}^1}|_{\mathcal{H}_1}=(t+1)h-(t+1)\sum_{i\in I}e_{i}-t\sum_{i\notin I}e_i.
$$
The statement now follows from Lemma \ref{bli for pullbacks} and \cite[Lemma 4.1]{BDPn+3}.
\end{proof}

\begin{remark}
If a divisor of the form \eqref{pulledback-cones} has an interpretation as a bilinear join, then a stronger base locus inequality may be found as an application of Lemma \ref{recipe-BLI-bilinear joins}; see for instance 
\ref{prop-bli-x346}, part (c), below, concerning the divisor in $X^{3,4}_6$ of the form \eqref{pulledback-cones} with $t=1$.
\end{remark}


\section{Birational geometry of \texorpdfstring{$X^{n,n+1}_{n+2}$}{X{n,n+1}{n+2}}}\label{section-Xn,n,n+2}
In this section, we consider the case of $X^{n,n+1}_{n+2}$ --- that is, the blowup of $\PP^n \times \PP^{n+1}$ in a set of $n+2$ general points.  We start by computing its effective cone and its movable cone. We then show that it is log Fano, in particular a Mori dream space.

\subsection{The effective and the movable cones}
We will consider divisors  on $X=X^{n,n+1}_{n+2}$ of the following form
\begin{equation}\label{general-divisor-n+2}
D=d_1H_1+d_2H_2-\sum_{i=1}^{n+2}m_iE_i.
\end{equation}

\begin{theorem}\label{eff-cone-n+2}
The monoid of effective divisor classes on $X^{n,n+1}_{n+2}$ is generated by the following classes:
\begin{itemize}
\item[($G_0$)] $E_i$,  for every $i\in\{1,\dots, n+2\}$,
\item[($G_1$)] $H_1-\sum_{i\in I}E_i$, for every $I\subset \{1,\dots, n+2\}$ with $|I|=n$,
\item[($G_2$)] $H_2-\sum_{i\in I}E_i$, for every $I\subset \{1,\dots, n+2\}$ with $|I|=n+1$.
\end{itemize}
In particular, the effective cone $\Eff(X^{n,n+1}_{n+2})$ has the following defining inequalities:
\begin{itemize}
\item[($\mathcal{I}_0$)] $d_j\ge0$, for every $j\in\{1,2\}$, 
\item[($\mathcal{I}_1$)]  $d_1+d_2\ge m_i$, for every $i\in\{1,\dots, n+2\}$,
\item[($\mathcal{I}_2$)]  $nd_1+(n+1)d_2\ge \sum_{i\in I}m_i$, for every $I\subseteq \{1,\dots, n+2\}$ with $|I|\in\{n+1,n+2\}$.
\end{itemize}
\end{theorem}
\begin{proof}
Given a divisor $D$ of the form \eqref{general-divisor-n+2}, we will prove the following claims:
\begin{itemize}
\item If $D$ is effective then it satisfies the inequalities $(\mathcal I_0)$, $(\mathcal I_1)$ and $(\mathcal I_2)$.
\item If $D$ satisfies $(\mathcal{I}_0)$, $(\mathcal{I}_1)$ and $(\mathcal{I}_2)$, then $D$ is a positive integer linear combination of divisors of type $(G_0)$, $(G_1)$ and $(G_2)$.
\item If $D$ is a positive linear combination of divisors of type $(G_0)$, $(G_1)$ and $(G_2)$, then $D$ is effective.
\end{itemize}

The third claim is obvious, since all divisors of type $(G_i)$, with $i=0,1,2$, are effective.

The first claim follows from the observation that the following curve classes in $N_1(X)_\mathbb{R}$,  that correspond respectively  to the inequalities $(\mathcal{I}_0)$, $(\mathcal{I}_1)$ and $(\mathcal{I}_2)$, are moving:
\begin{itemize}
\item[($\mathcal{I}_0$)] $h_j$,
\item[($\mathcal{I}_1$)]  $h_1+h_2-e_i$,
\item[($\mathcal{I}_2$)]  $nh_1+(n+1)h_2 -\sum_{i\in I}e_i$.
\end{itemize}
In fact, given any point that lies off the exceptional divisors and the strict transforms of the bilinear joins of the points $p_i$, there is a curve of each of the listed classes passing through the point. For the first two types of classes this is clear; for the third, the existence follows from Corollary \ref{corollary-ratcurvefewerpoints} (b).

We now prove the second claim, adapting to this case the simple idea used in \cite[Lemma 4.24]{CT06}. Assume first that $m_i>0$ for every $i\in\{1,\dots,n+2\}$. Consider a table with $d_2+d_1$ rows and $n+1$ columns. Leaving the last entry of each of the bottom $d_1$ rows blank, we may fill  the remaining $(n+1)d_2+nd_1$ entries each with an exceptional divisor starting from the first column and filling it top to bottom, then filling the second column etc. as follows: start by writing $m_1$ times $E_1$, then $m_2$ times $E_2$, etc. until we will have placed $E_{n+2}$ $m_{n+2}$ times. Condition $(\mathcal{I}_2)$ guarantees that this operation is possible, moreover  if the inequality holds strictly we shall leave blank the remaining entries of the table. 
Condition $(\mathcal{I}_1)$ guarantees that on each row there are no repetitions. 
This yields a decomposition of $D$ as a sum of divisor classes of type  $(G_1)$ and $(G_2)$ as follows:  the $d_2$  top  rows give rise to $d_2$ divisors of the form $H_2-\sum_{i\in I_2}E_i$,  with $|I_2|\le n+1$, and, similarly,  the $d_1$ bottom rows give rise to $d_1$ divisors of the form $H_1-\sum_{i\in I_1}E_i$, with $|I_1|\le n$. By construction, the sum  of these $d_2+d_1$ divisors is linearly equivalent to $D$. 

Assume now that $m_i\le0$ for some $i$.
Let  $I^-\subset \{1,\dots,n+2\}$ be the subset such that $m_i\le0$ for  every $i\in I^-$. We can write $D=D'-\sum_{i \in I^-}m_iE_i$ and, arguing as above, show that $D'$ decomposes as a sum of divisors of the form $(G_1)$ and $(G_2)$, so that $D$ is a sum of divisors of type $(G_0)$, $(G_1)$ and $(G_2)$.
This completes the proof of the second claim and of the theorem.
\end{proof}
   
\begin{theorem}\label{mov-cone-n+2}
The movable cone $\Mov(X^{n,n+1}_{n+2})$ is cut out by the inequalities $(\mathcal{I}_0)$, $(\mathcal{I}_1)$ and $(\mathcal{I}_2)$ of the effective cone and by the following inequalities:
\begin{itemize}
\item[($\mathcal{II}_0$)] $m_i\ge0$, for every $i\in\{1,\dots,n+2\}$, 
\item[($\mathcal{II}_1$)]  $(n-1)d_1+nd_2 \ge \sum_{i\in I}m_i$, for every $I\subset\{1,\dots, n+2\}$ with $|I|=n$,
\item[($\mathcal{II}_2$)]  $n(d_1+d_2)\ge \sum_{i\in I}m_i$, for every $I\subset\{1,\dots, n+2\}$ with $|I|=n+1$.
\end{itemize}
\end{theorem}
\begin{proof}
Theorem \ref{eff-cone-n+2} implies that an effective divisor is not movable if and only if it contains (at least) one of the divisors of type $(G_0), (G_1)$ or $(G_2)$ in its divisorial fixed part. 
Equivalently, a divisor is movable if and only if it is effective and it does not have any of the above fixed divisors in its base locus. 
The effectivity condition is equivalent to the three inequalities $(\mathcal{I}_0)$, $(\mathcal{I}_1)$, $(\mathcal{I}_2)$, by the second part of Theorem \ref{eff-cone-n+2}. Moreover, an exceptional divisor $(G_0)$ in not in the base locus of $D$ if and only if $(\mathcal{II}_0)$ is satisfied. 
Finally, a fixed divisor of the form $(G_i)$ in not in the base locus of $D$ if and only if $(\mathcal{II}_i)$ is satisfied; this follows from Proposition \ref{BLI pulledback hyperplanes}(b)  for $i=1$ and from Proposition \ref{BLI for bilinear spans}(b) for $i=2$.
\end{proof}

\begin{question}
Is the movable cone $\Mov(X^{n,n+1}_{n+2})$ generated by the following classes? 
\begin{itemize}
\item $M_{1,0,|I|}=H_1-\sum_{i\in I}E_i$, for every $I\subset \{1,\dots, n+2\}$ with $0\le |I|\le n-1$,
\item $M_{0,1,|I|}=H_2-\sum_{i\in I}E_i$, for every $I\subset \{1,\dots, n+2\}$ with $0\le |I|\le n$,
\item $M_{k,0,\epsilon}=kH_1-k\sum_{i\notin I\cup\{j\}}E_i- (k-1)\sum_{i\in I}E_i-\epsilon E_j$, for every $2\le k\le n$, $I\subset \{1,\dots, n+2\}$ with $|I|=k+1$, $j\notin I$ and $\epsilon\in\{0,1\}$,
\item $M_{0,k}=kH_2-k\sum_{i\notin I}E_i- (k-1)\sum_{i\in I}E_i$, for every $2\le k\le n+1$, $I\subset \{1,\dots, n+2\}$ with $|I|=k+1$,
\item $M_{k,1,\epsilon}=k H_1+H_2-(k+1)\sum_{i\notin I\cup\{j\}}E_i- k\sum_{i\in I}E_i-\epsilon E_j$, for every $1\le k\le n$, $I\subset \{1,\dots, n+2\}$ with $|I|=k+1$, $j\notin I$ and $\epsilon\in\{0,1\}$,
\item $M_{1,k}=H_1+kH_2-(k+1)\sum_{i\notin I}E_i- k\sum_{i\in I}E_i$, for every $1\le k\le n$, $I\subset \{1,\dots, n+2\}$ with $|I|=k+2$.
\end{itemize}
One can check explicitly that each of these classes is movable by decomposing it into effective classes in multiple ways with no common fixed component. 
For $n \leq 6$ we used Normaliz \cite{normaliz} to compute the generators of the cone cut out by the inequalities $(\mathcal{I}_i)$ and $(\mathcal{II}_i)$; these computations are contained in the files {\tt X(n)(n+1)(n+2)-Mov.*}, and the outputs show that in each case the resulting cone is spanned by the list of classes above. So these classes do indeed generate the movable cone for $n \leq 6$. It would be interesting to know if the same holds for all $n$.
\end{question}
\subsection{The log Fano property}
We now give a conceptual explanation for the finiteness of the cones computed in the previous subsection: namely, the variety $\ourproductblowuptwo$ has the log Fano property. As mentioned in Section \ref{section-preliminaries}, Birkar--Cascini--Hacon---McKernan showed that log Fano varieties are Mori dream spaces, which in particular implies that their effective and movable cones must be rational polyhedral. 

\begin{theorem}\label{theoremXn,n+1,n+2-logFano}
The variety $X^{n,n+1}_{n+2}$ is log Fano, for every $n\ge1$.
\end{theorem}
 In order to prove Theorem \ref{theoremXn,n+1,n+2-logFano}, we exhibit an effective divisor $\Delta$ such that the pair $(X,\Delta)$ is log Fano.

Consider general divisors $D_1$ and $D_2$ of class
\begin{align*}
D_1&=d_1H_1-m_1\sum_{i=1}^{n+2}E_i\\
&:= \left(3(n+1)^2(n+2)+n(n+2)^2 \right)H_1-\sum_{i=1}^{n+2} \left( 2n^2(n+2)+2n(n+1)^2-1 \right)E_i,\\
D_2&=d_2H_2-m_2\sum_{i=1}^{n+2}E_i\\
&:=\left(3(n+1)(n+2)^2+(n+1)^2(n+3)\right)H_2-\sum_{i=1}^{n+2} \left(2n(n+2)^2+2n(n+1)(n+3)-1 \right)E_i,
\end{align*} 
and set
\begin{equation}\label{divisorDelta}
\Delta=\frac{1}{4(n+1)(n+2)}(D_1+D_2).
\end{equation}

\begin{theorem}\label{theoremXn,n+1,n+2-logFanoexplicit}
Let $\Delta$ be the divisor on $X^{n,n+1}_{n+2}$ defined  
in \eqref{divisorDelta}. Then  $-K_X-\Delta$ is ample on $X^{n,n+1}_{n+2}$ and the pair $(X,\Delta)$ is klt.
\end{theorem}
 Let us give some explanation of how the divisors appearing in $\Delta$ were found. Our strategy was to choose an effective divisor of the form $\Delta=\epsilon \left(D_1+D_2\right)$, with $D_1$ and $D_2$ divisors pulled back from the two factors and $\epsilon$ a positive coefficient, in such a way that $-K-\Delta$ is ample and the discrepancies of $(X,\Delta)$ are greater than $-1$. These conditions translate to affine-linear inequalities on $\epsilon$ and the coefficients of the $D_i$. Computer experiments for small values of $n$ suggests plausible candidates for $D_1$, $D_2$ and $\epsilon$, and the proof of the theorem then consists in checking that these candidates do indeed give a log Fano pair for all $n$.

Recall, from  Notation \ref{notation:fibres}, the following sets of subvarieties of $\PP^n\times\PP^{n+1}$ obtained by taking fibres over  linear spans of points under the first and second projections, respectively, and their intersections:
\begin{align}\label{notation:ABC}
A&=\{L^1_I\times\PP^{n+1}: I\subset\{1,\dots,n+2\},1\le |I|\le n\},\notag\\
B&=\{\PP^{n}\times L^2_J: J\subset\{1,\dots,n+2\}, 1\le |J|\le n+1\}.\\
C&=\{L^1_I\times L^2_J: I,J\subset\{1,\dots,n+2\},1\le |I|\le n,  1\le |I|\le n+1\}.\notag
\end{align}
Notice that if $I=J$, then $L^1_I\times L^2_I=\BL(p_i:i\in I)$ is the bilinear span of the points $\{p_i:i\in I\}$.
By abuse of notation, we will use the same symbols for the strict transforms in $X$ of the elements of $A$, $B$ and $C$.

\begin{lemma}
The set $A\cup B\cup C$ is closed under intersection in $X$.
\end{lemma}
\begin{proof}
It is enough to show that $A$ and $B$ are closed under intersection, by the construction of $C$.

That $B$ is closed under intersection follows because  $\{L^2_J: 1\le |J|\le n+1\}$ is. In fact $n+2$ points in general position in $\PP^{n+1}$ form a set of homogeneous coordinate points, so that the $L^2_J$'s are coordinate linear subspaces of $\PP^{n+1}$. Such set is closed under intersection.

We prove the claim for $A$. Let $L^1_{I_1}\times\PP^{n+1}$ and $L^1_{I_2}\times\PP^{n+1}$ be two distinct elements, with $|I_i|=:a_i+1$.

Case 1. Assume that $I_1\cap I_2=\emptyset$. In this case $a+a'< n$. In fact, since $I_1,I_2\subset\{1,\dots,n+2\}$, then $|I_1|+|I_2|\le n+2$, i.e. $a+a'\le n$. If $a+a'=n$, we obtain the inequality $0<k_a(D_1)+k_{a'}(D_1)=(n+1)m_1-nd_1$, which contradicts the effectivity condition $(\mathcal{I}_2)$ that must be satisfied by $D_1$ by Theorem \ref{eff-cone-n+2}.
Hence it must be that $a+a'<n$. In this case, by the generality of the points, we conclude that $L_{I_1\cap I_2}=\emptyset$.

Case 2.  Assume that $I_1\cap I_2\ne\emptyset$. Notice that the obvious identities 
$I_i=(I_i\setminus I_j) \cup (I_i\cap I_j)$, for $\{i,j\}=\{1,2\}$, give
\begin{equation}\label{linear-intersections}L_{I_1}\cap L_{I_2}=\langle L_{I_2\setminus I_2}\cap L_{I_2\setminus I_2}, L_{I_1\cap I_2}\rangle.\end{equation}
Since $D_1$ is effective, using the inequality $(\mathcal{I}_1)$ of Theorem \ref{eff-cone-n+2}, we have that $k_{I_1}(D_1)\le k_{I_1\setminus I_2}(D_1)$ and  $k_{I_2}(D_1)\le k_{I_2\setminus I_1}(D_1)$. This implies that the products $L_{I_1\setminus I_2}\times\PP^{n+1}$ and $L_{I_2\setminus I_1}\times\PP^{n+1}$ are elements of $A$. Using Case 1,  since $(I_1\setminus I_2)\cap (I_2\setminus I_1)=\emptyset$, we obtain $L_{I_1\setminus I_2}\cap L_{I_2\setminus I_1}=\emptyset$ and so \eqref{linear-intersections} gives $L_{I_1}\cap L_{I_2}= L_{I_1\cap I_2}.$
\end{proof}

A log resolution $\phi: \widetilde{X}\to X$ of $(X,\Delta)$ is obtained by subsequently blowing-up the elements of $A\cup B\cup C$, following the inclusion order. The fact that $A\cup B\cup C$ is closed under intersection guarantees that, at every step of the sequence of blow-ups, the strict transforms of elements of the set of equal dimension are disjoint.

Call $\widetilde{D}_1$ and $\widetilde{D}_2$ the strict transforms of $D_1$ and $D_2$ respectively, under this composition of blow-ups. 
We will show that the union of $\widetilde{D}_1$, $\widetilde{D}_2$ and of the exceptional divisors on $\widetilde{X}$ is simple normal crossing. 
To prove the claim, we need the following result.

\begin{proposition}\label{basepoint-freeness}
 The divisors $\widetilde{D_1}$ and $\widetilde{D_2}$ are basepoint-free on $\widetilde{X}$, and in particular general members are smooth.
\end{proposition}
  \begin{proof}
We show the statement for $\widetilde{D_1}$; the proof for $\widetilde{D_2}$ is similar.

Set $X:=X^{n,n+1}_{n+2}$ and consider the sets of fibres and intersections of fibres $A,B$ and $C$ defined in \eqref{notation:ABC}, see also Notation \ref{notation:fibres}.
Consider the subsets 
\begin{align*} 
A_0& :=\{L_I^1\times\PP^{n+1}\in A: |I|=1\}\subset A\\ 
B_0&:=\{\PP^n\times L_J^2\in B: |J|=1\}\subset B\\
C_0&:=\{L^1_I\times L^2_J \in C: |I|=|J|=1\}\subseteq C
\end{align*}
Let  $\overline{\phi}:\overline{X}\to X$ be the subsequent blow-up of $X$ along the elements of $A_0\cup B_0\cup C_0$ following the inclusion order, i.e. first blowing up $C_0$, then $A_0$ and $B_0$. Recall that  $\pi_1 \colon X^{n,n+1}_{n+2}\to \PP^n$ denotes  the natural projection, cf. Notation \ref{notation:classes}.
By the universal property of the blow-up, $\pi_1\circ\overline{\phi}$ factors through $X^{n,0}_{n+2}\to \PP^n$, the blow-up of $\PP^n$ at the projections of the $n+2$ points of $\PP^n\times\PP^{n+1}$. Call $\overline{\pi}_1$ the corresponding map, see  Diagram \eqref{diagram2}.

Now, let $\widetilde{\phi}:\widetilde{X}\to  \overline{X}$ be the blow-up along the strict transforms in $\overline{X}$ of the remaining elements of $A\cup B \cup C$, following the inclusion order, and consider the composition 
$\overline{\pi}_1\circ\widetilde{\phi}:\widetilde{X}\to X^{n,0}_{n+2}$.
Denote with $E_\alpha, E_\beta, E_\gamma$ the exceptional divisor of elements $\alpha\in A$, $\beta \in B$ and $\gamma\in C$ respectively.

If $\alpha \in A$, there is an index set $I\subset\{1,\dots,n+2\}$ of cardinality $|I|\le n$ such that  we can write $\alpha=\alpha^n\times\PP^{n+1}=L^1_I \times \PP^{n+1} \subset X^{n,n+1}_{n+2}$, where $\alpha^n=L^1_I\subset X^{n,0}_{n+2}$ is the strict transform on $X^{n,0}_{n+2}$ of the linear span of $|I|$ points of $\PP^n$.
The pullback under $\overline{\pi}_1\circ\widetilde{\phi}$ of $\alpha^n$ is the divisor:
\begin{equation}\label{lifting-exceptionals}
(\overline{\pi}_1\circ\widetilde{\phi})^\ast(\alpha^n)=E_\alpha +\sum_{\gamma: \gamma \subset \alpha}E_\gamma 
\end{equation}
Hence, again by the universal property of the blow-up, $\overline{\pi}_1\circ\widetilde{\phi}$ factors through the blow-up $\phi^n:\widetilde{X}^{n,0}_{n+2}\to X^{n,0}_{n+2}$ along all $L^1_I$'s, with $E_{\alpha^n}$ exceptional divisors.  Call $\widetilde{\pi}_1:\widetilde{X}\to \widetilde{X}^{n,0}_{n+2}$  the so obtained map, see Diagram \eqref{diagram2}.
\begin{equation}\label{diagram2}\begin{tikzcd}
\widetilde{X}\arrow{r}{\widetilde{\phi}} \arrow{rd} \arrow{d}{\widetilde{\pi}_1} 
& \overline{X} \arrow{r}{\overline{\phi}} \arrow{rd} \arrow{d}{\overline{\pi}_1} & X \arrow{d}{\pi_1} \\
\widetilde{X}^{n,0}_{n+2} \arrow{r}{\phi^n}&   X^{n,0}_{n+2} \arrow{r}{} & \PP^n
\end{tikzcd}
\end{equation}

Now, consider the divisor $D_1^n$ on $X^{n,0}_{n+2}$ defined by the same formula as $D_1$, that is it is the strict transform on $X^{n,0}_{n+2}$ of a general hypersurface of $\PP^n$ of degree $d_1$ and passing through $n+2$ general points with assigned multiplicities $m_1,\dots,m_{n+2}$.
Let $\widetilde{D}^n_1$ denote the strict transform of $D^n_1$ under $\phi^n$.
A hypersurface of $\PP^n\times\PP^{n+1}$ of degree $(d_1,0)$ and multiplicities $m_1,\dots,m_{n+2}$ at $n+2$ points is a fibre over a hypersurface of $\PP^{n+1}$ of degree $d_1$ and multiplicities $m_1,\dots,m_{n+2}$ at $n+2$ points via the projection $\PP^n\times\PP^{n+1}\to\PP^n$. Hence
\begin{equation}\label{lifting-strict-transforms}\widetilde{D}_1=\widetilde{\pi}_1^\ast\widetilde{D}^n_1.\end{equation}
Results contained in  \cite[Theorem 2.1 and Remark 2.5]{DP19} show that the linear system of $\widetilde{D}^n_1$ is basepoint free; hence the same property holds for the linear system of $\widetilde{D}_1$. In particular,  by Bertini's Theorem, a general member is smooth. 
\end{proof}

Let $E_A, E_B, E_C$ denote the union of all exceptional divisors of elements of $A, B ,C$ respectively, and let $E=E_A\cup E_B\cup E_C$.

\begin{corollary}
In the above notation, the union of divisors 
$\widetilde{D}_1\cup \widetilde{D}_2\cup E$  is simple normal crossing in $\widetilde{X}$.
\end{corollary}
\begin{proof}
That the divisors $\widetilde{D}_1$ and $\widetilde{D}_2$ are smooth follows from Proposition \ref{basepoint-freeness}; all exceptional divisors are smooth since they arise from blowups along smooth centres, and the union $E$ is snc. 

The strict transforms $\widetilde{D}_1$ and $\widetilde{D}_2$ intersect transversely since they are pull-backs of divisors from different factors, i.e. $\widetilde{D}_1=(\tilde{\pi}_1)^\ast \widetilde{D}^n_1$, as in \eqref{lifting-strict-transforms}, and $\widetilde{D}_1=(\tilde{\pi}_2)^\ast \widetilde{D}^n_2$ similarly.

For the same reason, $\widetilde{D}_1$ and $E_\beta$ intersect transversely for every $\beta \in B$ so $\widetilde{D}_1\cup E_B$ is snc. Similarly $\widetilde{D}_2\cup E_A$ is snc.

Since $\widetilde{D}^n_1 \cup \bigcup_{\alpha \in A} E_{\alpha^n}$ is snc in $\widetilde{X}^{n,0}_{n+2}$, see \cite[Corollary 2.10]{DP19}, using \eqref{lifting-exceptionals} and \eqref{lifting-strict-transforms} we obtain that $\widetilde{D}_1\cup E_A\cup E_C$ is snc in $\widetilde{X}$. Similarly $\widetilde{D}_2\cup E_B\cup E_C$ is snc.
\end{proof}

We now describe explicitly the class of the strict transforms of general pull-back divisors $D_1$ and $D_2$ on $\tilde{X}$.
\begin{proposition}\label{prop:stricttransform}
\begin{enumerate}
\item Let $D_1=d_1H_1-\sum_{i=1}^{n+2}m_iE_i$ be a pullback divisor.
Then the strict transform $\widetilde{D}_1$ via the blow-up of $X$ along $A\cup B\cup C$ has the following class:
$$
\widetilde{D}_1=d_1H_1-\sum_{i=1}^{n+2}m_i E_i-\sum_{\alpha\in A}\kappa_\alpha(D_1)E_\alpha-\sum_{\gamma\in C}\kappa_\alpha(D_1)E_\gamma
$$
where $\alpha=L^1_I\times\PP^{n+1}$ and $\kappa_\alpha(D_1)=\max\{0,\sum_{i\in I}m_i-(|I|-1)d_1\}$, for  $I\subset\{1,\dots,n+2\}$.

\item Let $D_2=d_2H_2-\sum_{i=1}^{n+2}m_iE_i$ be a pullback divisor.
Then the strict transform $\widetilde{D}_2$ via the blow-up of $X$ along $A\cup B\cup C$ has the following class:
$$
\widetilde{D}_2=d_2H_2-\sum_{i=1}^{n+2}m_i E_i-\sum_{\beta\in B}\kappa_\beta(D_2)E_\beta-\sum_{\gamma\in C}\kappa_\beta(D_2)E_\gamma
$$
where $\beta=\PP^{n}\times L^2_J $ and $\kappa_\beta(D_2)=\max\{\sum_{i \in J}m_j-(|J|-1)d_2\}$, for  $J\subset\{1,\dots,n+2\}$.
\end{enumerate}
\end{proposition}
\begin{proof}
We prove the first statement; the second has a similar proof.

We use the same notation as in the proof of Proposition \ref{basepoint-freeness}: let $D^n_1$ be the divisor on $X^{n,0}_{n+2}$ with the same coefficients as $D_1$. For every element $\alpha \in A$  as in  \eqref{notation:ABC}, write $\alpha=\alpha^n\times\PP^{n+1}$ and consider the set 
$$
A^n=\{\alpha^n=L^1_I: 1\le |I|\le n\}.
$$
Using \cite[Theorem 2.6]{DP19}, we can describe the base scheme of $D^n_1$ as the formal sum
\begin{align*}
\textrm{BL}(D_1^n)&=\bigcup_{\alpha\in A: |I|>1} \kappa_{\alpha}(D_1) {\alpha}^n 
\end{align*}
where ${\alpha}^n$ is the strict transform in  $X^{n,0}_{n+2}$ of the linear span of $|I|$  points of $\PP^n$, for every $I\subset\{1,\dots,n+2\}$ with $2\le |I|\le n$.
In particular, the strict transform $\widetilde{D}^n_1$ of a general $D^n_1$ on $\widetilde{X}^{n,0}_{n+2}$ after blowing-up all the elements in $A^n$ with $|I|\ge2$ in order of increasing dimension  is  calculated in \cite[Notation 1.6]{DP19} to be the following
$$
\widetilde{D}^n_1=(\phi^n)^\ast D^n_1-\sum_{\alpha\in A: |I|>1}\kappa_{\alpha}(D_1)E_{\alpha^n},
$$
where the $E_{\alpha^n}$'s denote the exceptional divisors.
Now, using \eqref{lifting-strict-transforms}  we obtain the class of the strict transform $\widetilde{D}_1$: 
\begin{align*}
\widetilde{D}_1&=\widetilde{\pi}_1^\ast\widetilde{D}^n_1\\
&=
d_1H_1-\sum_{i=1}^{n+2}m_iE_i-
\sum_{\alpha \in A}\kappa_\alpha(D_1)\left(E_\alpha+\sum_{\gamma:\gamma\subset\alpha}E_\gamma\right)\\
\intertext{
where $E_\alpha$ and $E_\gamma$ denote the exceptional divisor on $\widetilde{X}$ of elements $\alpha\in A$ and $\gamma \in C$ respectively and $\kappa_\alpha(D_1)=\kappa_{\alpha^n}(D^n_1)=\max\{\sum_{i\in I}m_i-(|I|-1)d_1\}$.
Notice that for every $\gamma \in C$ there is a unique  $\alpha\in A$ with $\gamma\subset \alpha$ and such that $\gamma$ and $\alpha$ have the same image in $\PP^n$, and we have $\kappa_\gamma(D_1)=\kappa_\alpha(D_1)$ for such $\gamma$ and $\alpha$, hence, using \eqref{lifting-exceptionals}, we obtain 
}
\widetilde{D}_1&=d_1H_1-\sum_{i=1}^{n+2}m_iE_i-\sum_{\alpha\in A} \kappa_\alpha(D_1)E_\alpha-\sum_{\gamma\in C} \kappa_\gamma(D_1)E_\gamma.
\end{align*}
\end{proof}

We are ready to prove the main result of this section.

\begin{proof}[Proof of Theorem \ref{theoremXn,n+1,n+2-logFanoexplicit}]
We first prove that $-K_X-\Delta$ is ample, where 
\begin{align*}
-K_X=(n+1)H_1+(n+2)H_2-2n\sum_{i}E_i.    
\end{align*}
By Proposition \ref{proposition-nefcone} the Mori cone is generated by the curve classes $e_i$ and $l_j-e_i$. We have
\begin{align*}
(-K_X-\Delta)&\cdot e_i=
\frac{1}{2(n+1)(n+2)}>0\\
(-K_X-\Delta)&\cdot (l_1-e_i)=
\frac{1}{4(n+1)(n+2)}>0\\
(-K_X-\Delta)&\cdot (l_2-e_i)=
\frac{n-1}{4(n+1)(n+2)}>0    
\end{align*}
so $-K_X-\Delta$ is ample.

A log resolution of the pair $(X,\Delta)$ is obtained by blowing-up $A\cup B\cup C$. 
It remains to show that the pair is klt.
To do that, we use the explicit class of the strict transforms of $D_1$ and $D_2$ computed in Proposition \ref{prop:stricttransform}.
  \begin{itemize}
\item  Elements of $A$: take $\alpha=L^1_I\times \PP^{n+1}$, with $1\le|I|\le n$ and such that $\kappa_\alpha(D_1)>0$. The discrepancy of the exceptional divisor $E_\alpha$ is
$$
\discrep(\alpha)=
(n-|I|)-\frac{1}{4(n+1)(n+2)}(|I|m_1-(|I|-1)d_1).
$$
Notice that, since $d_1-m_1\ge 4(n+1)(n+2)$, as one can easily check, then we have the relation:
$$
\discrep(L^1_{I}\times \PP^{n+1})\ge \discrep(L^1_{I\setminus\{i\}}\times \PP^{n+1}).
$$
for every $i\in I$ and $|I|\ge 2$.
Therefore it is enough to prove that the discrepancy is $>-1$ in the initial cases $|I|=1,2$. We obtain
\begin{align*}
\discrep({p^1_i}\times \PP^{n+1})
&=(n-1)-\frac{1}{4(n+1)(n+2)}(m_1)\\ 
&=-\frac{6n+7}{4(n+1)(n+2)}\\ &>-1,\\
\discrep(
L^1_{\{i_1,i_2\}}\times \PP^{n+1})&=(n-2)-\frac{1}{4(n+1)(n+2)}(2m_1-d_1)\\ 
&=\frac{4n^2-n-8}{4(n+1)(n+2)}\\ &>-1. 
\end{align*}
Notice that if $\alpha\in A$ with $\kappa_\alpha(D_1)=0$, then $\discrep(\alpha)=n-|I|>-1.$
\item Elements of $B$: in this case the proof is similar. Let $\beta =\PP^n\times L^2_J$, with $1\le|J|\le n+1$ and such that $\kappa_\beta(D_2)>0$. The discrepancy of the exceptional divisor $E_\beta$ is 
$$
\discrep(\PP^{n}\times L^2_J)=(n+1-|J|)-\frac{1}{4(n+1)(n+2)}(|J|m_2-(|J|-1)d_2).
$$
Notice that, since $d_2-m_2\ge 4(n+1)(n+2)$, as one can easily check, then we have the relation:
$$
\discrep(\PP^{n}\times L^2_{J})\ge \discrep(\PP^{n}\times L^2_{J\setminus\{j\}}),
$$
for every $j\in  J$ and $|J|\ge2$.
Therefore it is enough to prove that the discrepancy is $>-1$ in the case $b=0,1$, i.e.
\begin{align*}
\discrep( \PP^{n}\times p^2_j)&=n-\frac{1}{4(n+1)(n+2)}(m_2)\\
&=-\frac{4n^2-6n-1}{4(n+1)(n+2)}\\
&>-1,\\
\discrep( \PP^{n}\times L^2_{\{j_1,j_2\}})
&=
(n-1)-\frac{1}{4(n+1)(n+2)}(2m_2-d_2)\\
&=-\frac{4n^2+n-9}{4(n+1)(n+2)}\\
\end{align*}
Notice that if $\beta\in B$ with $\kappa_\beta(D_1)=0$, then $\discrep(\beta)=n-|I|-1>-1.$

\item Elements of $C$: notice that the multiplicity of containment of $\alpha\cap\beta =L^1_I\times L^2_J$ in $D_1$ is $\kappa_\alpha(D_1)$ and in $D_2$ it is $\kappa_\beta(D_2)$. This gives 
$$
\discrep(\alpha\cap\beta)=
\discrep(\alpha)+\discrep(\beta)+1>-1,
$$
where the equality follows from an easy calculation and the inequality follows from the above two items.
\end{itemize}
\end{proof}

\section{Birational geometry of \texorpdfstring{$X^{2,3}_5$}{X{2,3}5}}
\label{section-X235}

In this section we tackle the first example in which bilinear secant varieties as defined in Section \ref{section-bilinear} play a key role, namely the variety $X^{2,3}_5$. First we compute the effective cone using the ``Cone Method'' introduced in \cite{GPP21}; we then prove that this variety is log Fano by exhibiting an explicit boundary divisor. 

\subsection{The effective and movable cones}

In this subsection we compute the effective and movable cones of $X^{2,3}_5$. We start by spelling out the relevant base locus lemmas that we need. 

For the rest of this section, we will write $X$ instead of $X^{2,3}_5$ as needed to simplify notation. We write $Q$ to denote the unique divisor on $X$ with class $2H_1-\sum_{i=1}^5 E_i$; geometrically it is the strict transform on $X$ of the preimage of the unique conic in $\PP^2$ passing through the projections of the points $p_1,\ldots,p_5$. In particular $Q$ is smooth and irreducible. We write $\overline{\BS_2(C)}$ to denote the strict transform on $X$ of the bilinear secant variety $\BS_2(C)$; by Proposition \ref{prop-secantmult} its class is $H_1+2H_2-2\sum_i E_i$.
\begin{proposition} \label{prop-bli-x235}
Let $D=d_1H_1+d_2H_2-\sum_{i=1}^5 m_i E_i$ be an effective divisor class on $X^{2,3}_5$. 
\begin{enumerate}
    \item[(a)] For any $\{i,j\} \subset \{1,\ldots,5\}$, the unique effective divisor with class $H_1-E_i-E_j$ is contained in the base locus $\Bs(D)$ with multiplicity at least $\max \{0,m_i+m_j-d_1-2d_2\}$.
    \item[(b)] For any $\{i,j,k\} \subset \{1,\ldots,5\}$, the unique effective divisor with class $H_2-E_i-E_j-E_k$ is contained in the base locus $\Bs(D)$ with multiplicity at least $\max \{0,m_i+m_j+m_k-2d_1-2d_2\}$.
    \item[(c)] The divisor $Q$ is contained in the base locus $\Bs(D)$ with multiplicity at least $\max\{0,\sum_i m_i - 2d_1-5d_2\}$.
    \item[(d)] The divisor $\overline{\BS_2(C)}$ is contained in the base locus $\Bs(D)$ with multiplicity at least $\max\{0,2\sum_k m_k - 5d_1-7d_2\}$ for each pair $\{i,j\} \subset \{1,\ldots,5\}$.
    \item[(e)] The divisor $\overline{\BS_2(C)}$ is contained in the base locus $\Bs(D)$ with multiplicity at least $\max\{0,\sum_i m_i -3d_1-3d_2\}$.
   
\end{enumerate}
\end{proposition}
\begin{proof}
Statement (a) follow from Proposition \ref{BLI pulledback hyperplanes} with $n=2$. For (b), note that the divisor in question is the strict transform of $\BL(p_i,p_j,p_k)$, so the claim follows from Proposition \ref{BLI for bilinear spans} with $n=2$ and $k=3$.  Statement (c) follows from Proposition \ref{prop-bli-pulledbackcone} with $n=2$, $t=1$, and $I$ equal to the empty set; (d) follows from Proposition \ref{prop-bli-bisecant} with $n=2$ and $k=2$.

For (e), we claim that irreducible curves of class $\gamma=3l_1+3l_2-\sum_i e_i$ sweep out a Zariski-dense open subset of $\overline{\BS_2(C)}$. Therefore any effective class $D$ such that $D \cdot \gamma <0$ must contain $\overline{\BS_2(C)}$ in its base locus; applying this repeatedly we get the desired statement. 

To prove the claim, in $\PP^3$ choose any smooth cubic passing through the images of the 5 points $p_1,\ldots,p_5$; such cubics sweep out a dense subset of $\PP^3$. The preimage of any such  cubic is isomorphic to $\PP^1 \times \PP^2$. So we are reduced to showing that for any 5 points in $\PP^1 \times \PP^2$ there is a 1-parameter family of curves of bidegree $(1,3)$ passing through all 5. Any such curve is the graph of a map from $\PP^1$ to $\PP^2$ given by 3 binary cubic forms. Such a set of 3 forms has 12 coefficients. Passing through a point imposes 2 conditions on these coefficients, therefore the family of curves passing through all 5 points has dimension $(12-1)-5\cdot2 = 1$. Since any such curve is a graph, it is smooth irreducible rational. The strict transforms of these curves sweep out a divisor in $X$. Since $\overline{\BS_2(C)} \cdot \gamma = -1$ any such curve is contained in $\overline{\BS_2(C)}$, so this divisor has to equal $\overline{\BS_2(C)}$.
\end{proof}
Finally, we also have an effectivity inequality that follows from Proposition \ref{prop-bli-bilinearjoins}:
\begin{proposition} \label{prop-bilinearjoins-x235}
On $X^{2,3}_5$ any effective divisor must satisfy
\begin{align*}
4d_1+5d_2 - m_i-m_j-\sum_{k=1}^5 m_k &\geq 0
\end{align*}
for each pair $\{i,j\} \subset \{1,\ldots,5\}$.
\end{proposition}
\begin{proof}
Put $k=1$ and $l=n=2$ in Proposition \ref{prop-bli-bilinearjoins}: in this case $\BJ(C,\BL(p_i,p_j))=X$ so for any effective divisor the multiplicity of containment must be zero. The Proposition therefore implies that
\begin{align*}
    (nk+k+l-1)d_1+(nk+2k+l-1)d_2-k\sum_{i=1}^{n+3}m_i-\sum_{j=1}^l m_{i_j} &\geq 0.
\end{align*}
Putting $k=1$, $l=n=2$ and writing $m_i$ and $m_j$ instead of $m_{i_1}$ and $m_{i_2}$, this gives exactly the inequality claimed.
\end{proof}

\begin{theorem}\label{theorem-effcone-x235}
The effective cone $\Eff(X^{2,3}_5)$ is generated by the divisor classes below.
\rowcolors{2}{gray!15}{white}
\emph{
\begin{longtable}{ccccccc}
$H_1$ & $H_2$ & $E_1$ & $E_2$ & $E_3$ & $E_4$ & $E_5$\\
\hline\hline
0 & 0 & 1 & 0 & 0 & 0 & 0\\
1 & 0 & -1 & -1 & 0 & 0 & 0\\
0 & 1 & -1 & -1 & -1 & 0 & 0\\
1 & 1 & -2 & -1 & -1 & -1 & -1\\
1 & 2 & -2 & -2 & -2 & -2 & -2\\
2 & 0 & -1 & -1 & -1 & -1 & -1
\end{longtable}}
\end{theorem}
Let us explain the conventions for reading this and subsequent tables. First, a row of the form $(d_1 \, d_2 \,  m_1 \cdots m_5)$ represents the divisor class $d_1H_1+d_2H_2+\sum_{i=1}^5 m_i E_i$: for example, the second row represents $H_1-E_1-E_2$. Second, the table should be read {\bf up to permutation}: for any class corresponding to a row $(d_1 \, d_2 \, m_1 \cdots m_5)$ of the table, all classes obtained by permutations of the $m_i$ are also included in the list of generators. So for example the third row of the table says that the list of generators  includes all classes of the form $H_2-\sum_I E_i$ where $I$ is a 3-element subset of $\{1,\ldots,5\}$. 

\begin{proof} 
We follow the ``Cone Method", described in detail in \cite{GPP21}. To begin, consider the following collections of inequalities:
\begin{itemize}
    \item The ``pullbacks" of the inequalities defining $\Eff(X^{2,3}_4)$ which are listed in Theorem \ref{eff-cone-n+2};
    \item The base locus inequalities for exceptional divisors from Lemma \ref{bli-exceptional};
    \item The base locus inequalities listed in Proposition \ref{prop-bli-x235};
    \item The effectivity inequality from Proposition \ref{prop-bilinearjoins-x235}.
\end{itemize}
By construction, the class of every irreducible effective divisor on $X$ other than $Q$, $\overline{BS_2(C)}$, the $E_i$, and pullbacks of extremal linear classes from the factors must satisfy all these inequalities. We use Normaliz to compute the generators of the cone cut out by these inequalities; these computations are contained in the files {\tt X235-eff-ineqs.*}.

Next we take the list of generators from the previous step, add to it the classes of $Q$ and $\overline{\BS_2(C)}$, exceptional divisors $E_i$ and pullbacks of extremal linear classes, and compute the new cone generated by all these classes. By construction, this cone will contain the effective cone, so if it is generated by effective classes, it must equal the effective cone. 

The files {\tt X235-eff-gens.*} compute the generators of the new cone: the output is the list of classes displayed in the statement of the theorem. All these are classes of effective divisors, as we now explain. The first row and its permutations give the classes of the exceptional divisors, while the second and third rows are the classes of pullbacks of linear spaces in the two factors. The last two rows are the classes of  $Q$ and $\overline{\BS_2(C)}$ respectively. For the fourth row, observe that $H_1+H_2-\sum_iE_i -E_j$ is the class of the strict transform of a hypersurface of bidegree $(1,1)$ passing through all the points $p_i$ and singular at the point $p_j$. A dimension count shows that the linear system of such hypersurfaces has positive expected dimension, so this class is effective.
\end{proof}

\begin{remark}\label{remark-extremalrays-X235}
It is interesting to observe that all extremal rays, except for those of degree $(1,1)$, are spanned by fixed divisors: exceptional divisors, pull-backs of fixed divisors from either factors, and  the proper transform of the bilinear secant variety $\BS_2(C)$.

The class $H_1+H_2-\sum_{i=1}^5E_j$ is a pencil whose elements can be described as \emph{bilinear cones} over the divisors of bidegree $(1,1)$ through four points of $\PP^1\times\PP^2$: see \cite[Theorem 3.6]{GPP21}.

\end{remark}
\begin{corollary}\label{corollary-X235movable}
    The movable cone $\Mov(X^{2,3}_5)$ is generated by the divisor classes below.
\rowcolors{2}{gray!15}{white}
\emph{
\begin{longtable}{ccccccc}
$H_1$ & $H_2$ & $E_1$ & $E_2$ & $E_3$ & $E_4$ & $E_5$\\
\hline\hline
1 & 0 & 0 & 0 & 0 & 0 & 0\\
1 & 0 & -1 & 0 & 0 & 0 & 0\\
2 & 0 & -1 & -1 & -1 & -1 & 0\\
2 & 0 & -1 & -1 & -1 & 0 & 0\\
3 & 0 & -2 & -1 & -1 & -1 & -1\\
0 & 1 & 0 & 0 & 0 & 0 & 0\\
0 & 1 & -1 & 0 & 0 & 0 & 0\\
0 & 1 & -1 & -1 & 0 & 0 & 0\\
0 & 2 & -2 & -1 & -1 & -1 & -1\\
0 & 2 & -2 & -1 & -1 & -1 & 0\\
0 & 3 & -2 & -2 & -2 & -2 & -1\\
0 & 3 & -2 & -2 & -2 & -2 & 0\\
1 & 1 & -2 & -1 & -1 & -1 & -1\\
1 & 1 & -2 & -1 & -1 & -1 & 0\\
1 & 1 & -2 & -1 & -1 & 0 & 0\\
1 & 2 & -2 & -2 & -2 & -2 & -1\\
1 & 2 & -2 & -2 & -2 & -2 & 0
\end{longtable}}
\end{corollary}
\begin{proof}
The listed classes are exactly the generators of the cone computed by {\tt X235-eff-ineqs.*}. By construction, this cone contains the movable cone $\Mov(X^{2,3}_5)$. If all these classes are movable, then the cone they span must equal $\Mov(X^{2,3}_5)$.

We see that all the listed classes are movable as follows:
\begin{itemize}
    \item All but the last 5 lines are pullbacks of movable classes from $X^{2,0}_5$ or $X^{0,3}_5$ so these classes are movable.
    \item For line $-5$, we have already seen in the proof of Theorem \ref{theorem-effcone-x235} that the class $H_1+H_2-\sum_i E_i -E_j$ corresponds to a positive-dimensional linear system. Moreover, this class is a primitive generator of an extremal ray of the effective cone, so every divisor in the linear system is irreducible. So this linear system has no fixed component, hence it is movable.
    \item The remaining classes all have two decompositions into sums of effective classes that do not share any fixed component of codimension 1:
    \begin{align*}
        H_1+H_2-2E_1-\cdots-E_4 &= \left(H_1-E_1-E_2 \right) + \left(H_2-E_1-E_3-E_4 \right)\\
                                &= \left(H_1-E_1-E_3 \right) + \left(H_2-E_1-E_2-E_4 \right)\\
        H_1+H_2-2E_1-E_2-E_3    &= \left(H_1-E_1-E_2 \right) + \left(H_2-E_1-E_3\right)\\
                                &= \left(H_1-E_1-E_2 \right) + \left(H_2-E_1-E_3\right)\\
                                &= \left(H_1-E_1-E_3 \right) + \left(H_2-E_1-E_2\right)\\
 H_1+2H_2-2(E_1+\cdots E_4)-E_5 &= \left(H_1+2H_2-2(E_1+\cdots+E_5)\right)+E_5\\
                                &= \left(H_1+H_2-2E_1-E_2-\cdots-E_5 \right)\\ & \quad +\left(H_2-E_2-E_3-E_4 \right)\\
 H_1+2H_2-2(E_1+\cdots E_4)     &= \left(H_1+2H_2-2(E_1+\cdots E_4)-E_5 \right) + E_5\\
                                 &=\left(H_1-E_1-E_2\right)+\left( H_2-E_1-E_3-E_4\right)\\ & \quad +\left( H_2-E_2-E_3-E_4\right)                             
    \end{align*}
    So the base locus of each of these classes has codimension at least 2, as required. 
\end{itemize}
   \end{proof}

\subsection{The log Fano property}
As in the case of $\ourproductblowuptwo$ discussed in Section \ref{section-Xn,n,n+2}, we now show that for $\blowup{2}{3}{5}$ the finiteness of the effective cone is explained by the log Fano property. 
\begin{theorem}\label{theorem-x235-logfano}
  The variety $\blowup{2}{3}{5}$ is log Fano.
\end{theorem}
As before, we will write $X$ instead of $X^{2,3}_5$ where needed to simplify notation.
\begin{proof}
Let $C$ be the distinguished rational curve through $p_1,\ldots,p_5$, whose existence was shown in Corollary \ref{corollary-ratcurvefewerpoints}. Define the following divisors on $X$:
  \begin{itemize}
  \item  $\Delta_1$ is the divisor $Q$ with class $2H_1-\sum_i E_i$.
  \item  $\Delta_2$ is the divisor $\overline{\BS_2(C)}$ with class $H_1+2H_2-2\sum_i E_i$.
  \item $\Delta_3$ is the sum $\sum_{i,j,k} \Pi^2_{\{i,j,k\}}$ where $\Pi^2_{\{i,j,k\}}$ is the strict transform on $X$ of the pullback of the plane in $\PP^3$ through the images of $p_i$, $p_j$, and $p_k$. The class of $\Delta_3$ equals $10H_2-6 \sum_i E_i$.
  \end{itemize}
  To exhibit a log Fano structure on $X$, we will consider a divisor of the form
  \begin{align*}
    \Delta &= \epsilon_1 \Delta_1 + \epsilon_2 \Delta_2 +\epsilon_3 \Delta_3.   \end{align*}
  We claim that for suitable choices of the coefficients $\epsilon_1, \epsilon_2, \epsilon_3$, the pair $(X,\Delta)$ is klt and $-K_X-\Delta$ is ample.

  First we consider the ample condition. By Proposition \ref{proposition-nefcone}, a divisor on $\blowup{2}{3}{5}$ is ample if it has positive intersection with the curve classes $e_i$ and $l_j-e_i$ for $i=1,\ldots,5$ and $j=1,2$. For $\Delta$ as above we have
    \begin{align*}
      -K_X - \Delta &= (3-2\epsilon_1 -\epsilon_2)H_1+(4-2\epsilon_2-10\epsilon_3)H_2 - \left( 4-\epsilon_1-2\epsilon_2-6\epsilon_3 \right) \sum_i E_i
    \end{align*}
    and therefore:
\begin{align*}
   (-K_X-\Delta)\cdot e_i>0 &\iff
 4-\epsilon_1-2\epsilon_2-6\epsilon_3>0,\\
 (-K_X-\Delta)\cdot (l_1-e_i)>0 &\iff
 4-\epsilon_1-2\epsilon_2-6\epsilon_3<3-2\epsilon_1-\epsilon_2,\\
 (-K_X-\Delta)\cdot (l_2-e_i)>0 &\iff
 4-\epsilon_1-2\epsilon_2-6\epsilon_3 <4-2\epsilon_2-10\epsilon_3.
\end{align*}   
These three conditions can be satisfied simultaneously, for example by taking
\begin{align*}
\left(\epsilon_1,\epsilon_2,\epsilon_3 \right) &=\left( \frac12, \, \frac{10}{11}, \, \frac{1}{10} \right).
\end{align*}
We will now prove that for these values of the $\epsilon_i$ the resulting pair $(X,\Delta)$ satisfies the klt condition. For this, we need to construct a log resolution of the pair, and show that the discrepancies computed on this resolution satisfy the necessary inequalities. 

Since $X$ is already smooth, in order to construct a log resolution of the pair $(X,\Delta)$ we need to blow up along the locus where $\operatorname{Supp}(\Delta)$ fails to be simple normal crossing. This consists of the following strata:
\begin{itemize}
    \item The strict transforms of the subsets $\Pi_{\{i\}}^2$ which are the fibres of the projections of points $p_i$ to $\PP^3$. These have codimension 3, but 7 components of $\operatorname{Supp}(\Delta)$ intersect along each $\Pi_{\{i\}}^2$: the divisor $\Delta_2$, as well as each of the 6 divisors $\Pi^2_{\{i,j,k\}}$ where $j$ and $k$ are distinct.
    \item The strict transforms of the subsets $\Pi_{\{i,j\}}^2$ which are the preimages of the lines spanned by the projections of two points $p_i$ and $p_j$ to $\PP^3$. These have codimension 2, but 3 components of $\operatorname{Supp}(\Delta)$ intersect along each one, namely the divisors $\Pi^2_{\{i,j,k\}}$ where $k$ is an element of $\{1,2,3,4,5\} \setminus \{i,j\}$.
    \item The strict transforms of the bilinear spans $\BL(p_i,p_j)$ of two points $p_i$ and $p_j$. This has codimension 3, but 4 components of $\operatorname{Supp}(\Delta)$ intersect along it, namely the divisors $\Pi^2_{\{i,j,k\}}$ where $k$ is an element of $\{1,2,3,4,5\} \setminus \{i,j\}$ as well as the divisor $\Delta_2$.
    \item Finally, the curve $C$ is exactly the singular locus of the divisor $\Delta_2$.
\end{itemize}
To obtain our log resolution we blow up these subvarieties in the following order:
\begin{itemize}
    \item the pairwise disjoint subvarieties $\Pi_{\{i\}}^2$ for $i \in \{1,2,3,4,5\}$;
    \item the strict transforms of the bilinear joins $\BL(p_i,p_j)$ for $\{i,j\} \subset \{1,2,3,4,5\}$, which are pairwise disjoint after the previous blowups;
    \item the strict transforms of the subvarieties $\Pi_{\{i,j\}}^2$ for $\{i,j\} \subset \{1,2,3,4,5\}$, also pairwise disjoint after the first set of blowups;
    \item the strict transform of the curve $C$. 
\end{itemize}
We write $\rho \colon \widetilde{X} \arrow X$ to denote this composition of blowups.

We need to verify that $\rho$ does indeed give a log resolution of the pair $(X,\Delta)$. Since the $\Pi^2_{\{i,j,k\}}$ are pullbacks of linear spaces in $\PP^3$ and we blow up their intersections, the strict transforms of these components are smooth and pairwise disjoint after blowing up. Moreover each $\Pi^2_{\{i,j,k\}}$ intersects $\Delta_1$ transversely on $X$ already, so the same remains true after blowing up. So it remains to consider the intersections of the strict transform of $\Delta_2$ with other components. Let $\widetilde{\Delta_i}$ denote the strict transforms on $\widetilde{X}$ of the divisors $\Delta_i$ defined previously. We claim that the intersection of $\widetilde{\Delta_2}$ with any component of the other $\widetilde{\Delta_i}$ is transverse.

To prove the claim, first we consider the intersection of $\Delta_2$ with one of the divisors $\Pi^2_{i,j,k}$. 
As in equations (\ref{eqn-M1}) and (\ref{eqn-M2}) we can assume that the curve $C$ and the divisor $\BS_2(C)$ are defined by the maximal minors of the matrices $M_1$ and $M_2$ with $m=3$:
\begin{align*}
       M_1 &=
    \begin{pmatrix}
      x_0 & x_1  & y_0 & y_1 & y_2 \\
      x_1 & x_2  & y_1 & y_2 & y_3
    \end{pmatrix}\\
     M_2 &= 
    \begin{pmatrix}
      x_0 & y_0 & y_1\\
      x_1 & y_1 & y_2\\
      x_2 & y_2 & y_3
    \end{pmatrix}
\end{align*}
Moreover, by applying projective transformations of $\PP^2 \times \PP^3$ fixing $C$ and hence $\BS_2(C)$, we can assume that the divisor $\Pi^2_{\{i,j,k\}}$ is the strict transform of the subvariety defined by $y_1=y_2$. Direct computation then shows that in $\PP^2 \times \PP^3$ the intersection $\Delta_2 \cap \Pi^2_{\{i,j,k\}}$ is transverse outside the subset $\{p_i,p_j,p_k\}$; moreover, at these points, the projectivised tangent cones of the two varieties intersect transversely inside the projectivised tangent space. Blowing up all the $p_i$, therefore, the intersection $\Delta_2 \cap \Pi^2_{\{i,j,k\}}$ becomes transverse on $X$, and therefore remains so on $\widetilde{X}$.

Next we consider the intersection of $\widetilde{\Delta}_1 \cap \widetilde{\Delta}_2$ on $\widetilde{X}$. Consider the blowup $\varphi \colon Y \arrow X$ of $X$ along $C$: by the universal property of blowups, the blowup $\rho \colon \widetilde{X} \arrow X$ factors through $\varphi$, so it suffices to show that the strict transforms of $\Delta_1$ and $\Delta_2$ on $Y$ are transverse.
Again we can assume that $\Delta_2$ is given by the determinant of $M_2$, and that $\Delta_1$ is the strict transform on $X$ of the conic $\left\{x_1^2-x_0x_2=0\right\} \subset \PP^2 \times \PP^3$.
  Direct calculation shows that $\Delta_1$ and $\Delta_2$ intersect transversely away from $C$, so it is enough to show that the intersection of their strict transforms is transverse inside the exceptional divisor $\pi^{-1}(C)$. The group of automorphisms of $\PP^2 \times \PP^3$ fixing $C$ act transitively on $C$, so it suffices to consider the intersection inside one fibre $\pi^{-1}(p)$, for some chosen $p \in C$. Choosing $p=\left( [1,0,0],[1,0,0,0]\right)$, direct calculation shows that the intersection of projectivised tangent cones $\PP T_p \Delta_1 \cap \PP T_p \Delta_2$ inside $\PP T_p X$ is a quadric cone of corank 1, with vertex at the point $c=\PP T_p C$ corresponding to the tangent direction of the curve $C$ at $p$.

  The fibre of the blowup $\varphi \colon Y \arrow X$ over $p \in C$ is the projectivisation of the fibre of the normal bundle $(N_C)_p = (T_X)_p/(T_C)_p$. The map
  \begin{align*}
    \PP T_p X \rat \PP (N_C)_p
  \end{align*}
  is given by projection away from the point $c$. Since $\PP T_p \Delta_1 \cap \PP T_p \Delta_2$ has corank 1 with vertex at $c$, when we project away from $c$, the images of $\PP T_p \Delta_1$ and $\PP T_p \Delta_2$ intersect transversely in $\PP (N_C)_p$. This completes the proof of the claim.
  
  Finally we need to show that discrepancies of all exceptional divisors on this log resolution are strictly greater than $-1$. Computing these discrepancies is straightforward since $\Delta_1$ and all the components of $\Delta_3$ are nonsingular, while $\Delta_2$ is nonsingular outside $C$ and has multiplicity 2 along $C$.
  
For a subvariety $V$ which is blown up we have
\begin{align*}
\discrep(V) &= \left(\operatorname{codim}(V,X) -1 \right) - \operatorname{mult}_V\Delta
\end{align*}
which for the various kinds of blown up subvarieties gives
\begin{align*}
    \discrep(\Pi^2_{\{i\}}) &= 2-\epsilon_2-6\epsilon_3 \\
    \discrep(\BL(p_i,p_j)) &= 2 - \epsilon_2- 3\epsilon_3\\ 
    \discrep(\Pi^2_{\{i,j\}}) &= 1-\epsilon_2-\epsilon_3\\
    \discrep(C) &=3-\epsilon_1-2\epsilon_2. 
\end{align*}
Taking as before $\epsilon_1 = \frac12, \, \epsilon_2 = \frac{10}{11}, \, \epsilon_3 = \frac{1}{10}$, all these discrepancies are strictly larger than $-1$, as required. 
\end{proof}

\section{The effective cone of \texorpdfstring{$X^{3,4}_6$}{X{3,4}6}}\label{section-X346}

In a similar way to the case of $X^{2,3}_5$, we can compute the effective cone of $X^{3,4}_6$ using the Cone Method, with various base locus inequalities as inputs. 

We recall some notation as follows. The bilinear secant $\BS_2(C)$ is a 4-dimensional subvariety of $\PP^3 \times \PP^4$. Let $p_i$ be one of the points: then the strict transform of the bilinear join $\BJ(p_i,\BS_2(C))$ is a divisor in $X^{3,4}_6$, which we denote by $D_i$. It has class $H_1+2H_2-2\sum_k E_k -E_i$. Let $p^1_1,\ldots,p^1_6$ be the projections of the 6 points to $\PP^3$: then for each $i$ there is a unique quadric cone $Q$ passing through all of $p^1_1,\ldots,p^1_6$ and with vertex at $p^1_i$. Let $Q_i$ denote the strict transform on $X^{3,4}_6$ of the pullback of this quadric cone; it has class $2H_1-\sum_kE_k-E_i$.

\begin{proposition} \label{prop-bli-x346}
Let $D=d_1H_1+d_2H_2-\sum_i m_i E_i$ be an effective divisor class on $X^{3,4}_6$. Then
\begin{enumerate}
\item[(a)] For any $\{i,j,k,l\} \subset \{1,\ldots,6\}$, the unique effective divisor with class $H_2-E_i-E_j-E_k-E_l$ is contained in the base locus $\Bs(D)$ with multiplicity at least
\begin{align*}
    \max \{0, m_i+m_j+m_k+m_l-3d_1-3d_2 \}.
\end{align*} 
    \item[(b)] For each $i$, the divisor $Q_i$ is contained in the base locus $\Bs(D)$ with multiplicity at least
\begin{align*}
    \max \{0, \sum_k m_k + m_i - 4d_1-6d_2\}.
\end{align*}
\item[(c)] For each $i$, the divisor $D_i$ is contained in the base locus $\Bs(D)$ with multiplicity at least 
\begin{align*}
    \max \{0, 2 \sum_k m_k + m_i - 8d_1 - 10d_2\}.
\end{align*}
\end{enumerate}
\end{proposition}
\begin{proof}
Statement (a) follows from Proposition \ref{BLI for bilinear spans} with $n=3$ and $k=4$.

For (b), we claim that $Q_i$ is the bilinear join of $C$ and the fibre $\Pi^1_{\{i\}}\cong\PP^4$ over $p_1^i$. Accepting the claim, since the strict transform of $\Pi^1_{\{i\}}$ in $X^{3,4}_6$  is contained in the base locus of $D$ at least $\max\{0,m_i-d_2\}$ times, then the statement follows from Lemma \ref{recipe-BLI-bilinear joins}.
We now prove the claim. Let $q\in\PP^3\times\PP^4$ be any point which does not lie on $\Pi^1_{\{i\}}$. By Definition \ref{definition-bilinearspan}, the bilinear join of $q$ and $\Pi^1_{\{i\}}$ is
\begin{align*}
\BL(\Pi^1_{\{i\}},q)&=\bigcup_{p\in\Pi^1_{\{i\}}}L^1(p,q)\times L^2(p,q)\\
&=L^1(p_i,q)\times\bigcup_{p\in\Pi^1_{\{i\}}}L^2(p,q)\\
&=L^1(p_i,q)\times\PP^4.
\end{align*}
From this, we obtain the following characterisation of the bilinear join of $C$ and $\Pi^1_{\{i\}}$:
\begin{align*}
\BL(\Pi^1_{\{i\}},C)&=\bigcup_{q\in C}\BL(\Pi^1_{\{i\}},q)\\
&=\bigcup_{q\in C} L^1(p_i,q)\times\PP^4\\
&=\left(\bigcup_{q\in C} L^1(p_1,q)\right)\times\PP^4.
\end{align*}
The first factor of the above expression is the quadric cone of $\PP^3$ over the projection of the curve $C$ to $\PP^3$ with vertex $p_i^1$, hence the strict transform of the product is $Q_i$, which proves the claim.

Statement (c) follows from Proposition \ref{prop-bli-bilinearjoins}, setting $k=2$ and $l=1$.
\end{proof}

\begin{proposition} \label{prop-effectivityx346}
Any effective divisor class $D=d_1H_1+d_2H_2-\sum_i m_i E_i$ on $X^{3,4}_6$ must satisfy the following inequalities:
\begin{enumerate}
    \item[(a)] $5d_1 + 4d_2 - \sum_{i=1}^6 m_i  \geq 0$;
    \item[(b)] $6d_1+7d_2 - 2 \sum_{i \in I} m_i - \sum_{j \notin I} m_j  \geq 0$ for $I \subset \{1,\ldots,6\}, \ |I|=3$;
    \item[(c)] $11d_1+14d_2 - 3 \sum_{i=1}^6 m_i \geq 0.$
\end{enumerate}
\end{proposition}
\begin{proof}
To prove (a), start by choosing a smooth curve of degree 4 in $\PP^4$ through the 6 points $p^2_1,\ldots,p^2_6$; such curves cover $\PP^4$, and any such curve is isomorphic to $\PP^1$. The preimage of any such curve is isomorphic to $\PP^3 \times \PP^1$, so we are reduced to showing that curves of bidegree $(5,1)$ through 6 points cover $\PP^3 \times \PP^1$. This follows from Corollary \ref{corollary-ratcurvefewerpoints} (c). 

To prove (b), note that for 3 points $p_i$, $p_j$, and $p_k$, the bilinear join $\BJ(C,\BL(p_i,p_j,p_k))$ equals $X^{3,4}_6$. Applying Proposition \ref{prop-bli-bilinearjoins} with $k=1$, $l=3$, and $n=3$ gives the required inequality. 

For (c), we note that the third bilinear secant $\BS_3(C)$ also equals $X^{3,4}_6$. Then the claimed inequality follows from Proposition \ref{prop-bli-bisecant} with $n=k=3$.
\end{proof}
\begin{theorem}
    \label{theorem-effcone-x346}
The effective cone $\Eff(X^{3,4}_6)$ is generated by the divisor classes below.
\rowcolors{2}{gray!15}{white}
\emph{
\begin{longtable}{cccccccc}
$H_1$ & $H_2$ & $E_1$ & $E_2$ & $E_3$ & $E_4$ & $E_5$ & $E_6$\\
\hline\hline
0 & 0 & 1 & 0 & 0 & 0 & 0 & 0\\
1 & 0 & -1 & -1 & -1 & 0 & 0 & 0\\
0 & 1 & -1 & -1 & -1 & -1 & 0 & 0\\
1 & 2 & -3 & -2 & -2 & -2 & -2 & -2\\
2 & 0 & -2 & -1 & -1 & -1 & -1 & -1\\
1 & 1 & -2 & -2 & -1 & -1 & -1 & -1\\
2 & 1 & -2 & -2 & -2 & -2 & -2 & -2
\end{longtable}}
\end{theorem}
\begin{proof}
    The proof is very similar to that of Theorem \ref{theorem-effcone-x235}. Again we follow the ``Cone Method",
 using the following collections of inequalities:
\begin{itemize}
    \item The ``pullbacks" of the inequalities defining $\Eff(X^{3,4}_5)$ which are listed in Theorem \ref{eff-cone-n+2};
    \item The base locus inequalities for exceptional divisors from Lemma \ref{bli-exceptional}
    \item The base locus inequalities from Proposition \ref{prop-bli-x346}.  
    \item The effectivity inequalities from Proposition \ref{prop-effectivityx346}.
  
\end{itemize}
By construction, the class of every irreducible effective divisor on $X^{3,4}_6$ other than the $Q_i$, the $D_i$, the $E_i$, and pullbacks of extremal linear classes must satisfy all these inequalities. We use Normaliz to compute the generators of the cone cut out by these inequalities; these computations are contained in the files {\tt X346-eff-ineqs.*}.

Next we take the list of generators from the previous step, add to it the classes of the $Q_i$ and the $D_i$, exceptional divisors $E_i$, and pullbacks of extremal linear classes, and compute the new cone generated by all these classes. By construction, this cone will contain the effective cone, so if it is generated by effective classes, it must equal the effective cone. 

The files {\tt X346-eff-gens.*} compute the generators of the new cone: the output is the list of classes displayed in the statement of the theorem. All these are classes of effective divisors, as we now explain. The first row and its permutations give the classes of the exceptional divisors, while the second and third rows are the classes of pullbacks of linear spaces in the two factors. The fourth and fifth rows are the classes of the $D_i$ and $Q_i$ respectively. For the fourth row, in Theorem \ref{theorem-effcone-x235} we explained that the class $H_1+H_2-E_i-\sum_k E_k$ is effective on $X^{2,3}_5$. Taking cones over these divisors with vertex at $p_j$, we get effective divisors with the given class. Finally, the class $2H_1+H_2-2\sum_i E_i$ has positive expected dimension, so it is effective.
\end{proof}

\begin{remark}
As in the case of $X^{2,3}_5$, we can also obtain the movable cone $\Mov(X^{3,4}_6)$ as a byproduct of this method. The resulting cone has many types of extremal rays, so we do not reproduce the results here. The interested reader can consult the file {\tt X346-eff-ineqs.out}.
\end{remark}

\section{Final remarks and open questions}
\label{section-Xn,n,n+3}
In this final section, we make some remarks to indicate how the phenomena we observed in the case of $X^{2,3}_5$ and $X^{3,4}_6$
recur, with increasing complexity, in higher dimensions, and we pose some questions.

In the case $n=2$ we obtained a fixed divisor as the bilinear secant variety $\BS_2(C)$ of our distinguished rational curve. For $n=3$ this construction no longer gives a divisor, but we do obtain divisors by taking bilinear joins of $\BS_2(C)$ with points. To generalise this, fix $n$ and let $l$ and $k$ be natural numbers such that $3k+2l-2=2n$. Choose $l$ points among $p_1,\ldots,p_{n+3}$, say for simplicity $p_1,\ldots,p_l$. Then the bilinear join
\begin{align*}
    \BJ(p_1,\ldots,p_l,\BS_k(C))
\end{align*}
is a divisor in $X^{n,n+1}_{n+3}$. In a similar way to Propositions \ref{prop-bli-x235} and \ref{prop-bli-x346}, using the base locus inequalities from Proposition \ref{prop-bli-bilinearjoins} one can then prove:
\begin{proposition}
For $n,k,l$ as above, the divisors $\BJ(p_1,\ldots,p_l,\BS_k(C))$ defined above span extremal rays of $\Eff(X^{n,n+1}_{n+3})$.
\end{proposition}
For convenience, we will call any divisor obtained in this way a BSJ ("bilinear secant and join") divisor. As $n$ increases, the list of generators of the effective cone becomes increasingly complicated. On one hand, for larger $n$ the number of solutions of $3k+2l-2=2n$ grows, and hence there are more types of BSJ divisors spanning extremal rays. On the other hand, it is not clear which movable classes will appear; for each $n$, we will find among the generators of $\Eff(X^{n,n+1}_{n+3})$ the bilinear cone over extremal movable classes from $X^{n-1,n}_{n+2}$, as we did above, but already for $n=3$ a new movable class appeared that was not a cone over something of lower dimension. We ask the following:

\begin{question}
Is the effective cone $\Eff(X^{n,n+1}_{n+3})$ always spanned by classes of BSJ divisors, exceptional divisors, and extremal classes pulled back from the factors, together with the movable cone? By Theorems \ref{theorem-effcone-x235} and \ref{theorem-effcone-x346} the answer is affirmative for $n=2$ and $n=3$. 
\end{question}

Going further, we can ask whether the log Fano property proved in Theorem \ref{theorem-x235-logfano} generalises to all dimensions:
\begin{question}
    Is $X^{n,n+1}_{n+3}$ log Fano for every $n \geq 1$?
\end{question}
The answer is yes for $n=1$ by \cite[Theorem 6.10]{GPP21} and for $n=2$ by Theorem \ref{theorem-x235-logfano}. For $n=3$ we can identify a suitable boundary divisor $\Delta$ making $-K_X-\Delta$ ample; the difficulty is to show that a certain blowup of $X^{3,4}_6$ does indeed give a log resolution. We explained above that as $n$ increases, we obtain more and more fixed divisors spanning extremal rays of the effective cone $\Eff(X^{n,n+1}_{n+3})$. It seems likely that any suitable boundary divisor on $X^{n,n+1}_{n+3}$ will contain many of these fixed divisors as components of its support, meaning that the identification of a log resolution becomes more and more complicated. 

Finally, in small dimensions there are also some ``exceptional" examples $X^{n,n+1}_s$ with $s\geq n+4$ for which we do not know if they are Mori dream spaces or log Fano. We ask:
\begin{question}\label{question-LF-smalldim} 
Which of the following blow-ups are Mori dream spaces or log Fano?
\begin{align*}
X^{2,3}_6, \, X^{2,3}_7, \, X^{3,4}_7, \, X^{4,5}_8.    
\end{align*}
\end{question}


\bibliographystyle{alpha}
\bibliography{biblio}

{\small
  {\sc 
\noindent {\sc Dipartimento di Matematica, Universit\`a degli Studi di Trento, via Sommarive 14, I-38123 Povo di Trento (TN), Italy}

\noindent{\tt elisa.postinghel@unitn.it}
~\newline
    Department of Mathematical Sciences, Loughborough University, Epinal Way, Loughborough LE11 3TU, United Kingdom}

\noindent {\tt A.Prendergast-Smith@lboro.ac.uk}
}
\end{document}